\documentclass{amsart}
\usepackage{amssymb,amsmath,latexsym,amsfonts,amsthm,mathrsfs,lineno}
\usepackage{enumerate}
\newtheorem{thm}{Theorem}
\newtheorem{thmm}{Theorem}
\newtheorem{defn}{Definition}
\newtheorem{prop}{Proposition}

\newtheorem{cor}{Corollary}

\newcommand{\N}{\mathbb{N}}
\newcommand{\Z}{\mathbb{Z}}
\newcommand{\R}{\mathbb{R}}
\newcommand{\Q}{\mathbb{Q}}
\newcommand{\T}{\mathbb{T}}
\newcommand{\D}{\mathbb{D}}
\newcommand{\funct}[2]{#1 \longrightarrow #2}
\newcommand{\m}[1]{\textbf{#1}}
\newcommand{\mc}[1]{\widetilde{\textbf{#1}}}
\newcommand{\Aut}{\mathrm{Aut}}
\newcommand{\Stab}{\mathrm{Stab}}
\newcommand{\Age}{\mathrm{Age}}
\newcommand{\Ty}{\overrightarrow{\mathbb{T}}}
\newcommand{\Cy}{\textbf{S}(2)}
\newcommand{\Cz}{\textbf{S}(3)}
\newcommand{\Dz}{\overrightarrow{\mathbb{D}}}
\newcommand{\LO}{\mathrm{LO}}

\newcommand{\arrows}[3]{\longrightarrow {#1}^{#2}_{#3}}
\newcommand{\restrict}[2]{#1\mathbin{\upharpoonright} #2}
\newcommand{\la}[1]{\stackrel{#1}{\longleftarrow}}

\author{L. Nguyen Van Th\'e}

\address{Laboratoire d'Analyse, Topologie et Probabilit\'es, Universit\'e d'Aix-Marseille, B\^{a}timent Henri Poincar\'e, 
Cour A, Avenue de l'escadrille Normandie-Ni\'emen, 13397 Marseille Cedex 20, France}

\email{lionel@latp.univ-mrs.fr}
\title{More on the Kechris-Pestov-Todorcevic correspondence: Precompact expansions}


\date{March 2012}

\begin{document}

\begin{abstract}
In 2005, the paper \cite{KPT} by Kechris, Pestov and Todorcevic provided a powerful tool to compute an invariant of topological groups known as the \emph{universal minimal flow}. This immediately led to an explicit representation of this invariant in many concrete cases. However, in some particular situations, the framework of \cite{KPT} does not allow to perform the computation directly, but only after a slight modification of the original argument. The purpose of the present paper is to supplement \cite{KPT} in order to avoid that twist and to make it adapted for further applications.
\end{abstract}

\maketitle

\section{Introduction}

The article \cite{KPT}, published in 2005 by Kechris, Pestov and Todorcevic, established a surprising correspondance between structural Ramsey theory and topological dynamics. As an immediate consequence, it triggered a new interest for structural Ramsey theory, as witnessed by the papers \cite{Ne} by Ne\v set\v ril, \cite{NVT1} by the present author, \cite{So0}, \cite{So1}, \cite{So2} by Soki\'c, and \cite{J} by Jasi\'nski. More recently, it also motivated a more detailed investigation of the connection between combinatorics and Polish group actions, an aspect which is visible in, for example, \cite{B} by Barto\v sov\'a, \cite{BP} by Bodirsky-Pinsker, \cite{MT} by Melleray-Tsankov, and \cite{M} by Moore. 
%Of course, it is not the first time that the idea of relating combinatorial behavior of classes of finite objects and group theoretic properties of Polish groups (in particular, closed subgroups of $S_{\infty}$, the permutation group of $\N$ equipped with the pointwise convergence topology) turns out to be fruitful, and there is a lot of beautiful results building on the duality between automorphism groups of countable first-order structures and closed subgroups of $S_{\infty}$. It is, however, probably fair to say that \cite{KPT} corresponds to a particularly successful instance. 
Precisely, \cite{KPT} provided an extremely powerful tool to compute an invariant known as the \emph{universal minimal flow}, and immediately led to an explicit representation of this invariant in many concrete cases. However, in some particular situations, the framework of \cite{KPT} does not allow to perform the computation directly, but only after a slight modification of the original argument. The purpose of the present paper is to supplement \cite{KPT} in order to avoid that twist and to make it adapted for further applications.

In order to describe precisely in which sense \cite{KPT} is generalized, we proceed to a synthetic overview of some of the main results it contains. Our main reference here is \cite{KPT} itself. In what follows, $\N$ denotes the set $\{ 0, 1, 2, \ldots\}$ of natural numbers, and for a natural number $m$, $[m]$ will denote the set $\{ 0, \ldots, m-1\}$. We will assume that the reader is familiar with the concepts of first order logic, first order structures, Fra\"iss\'e theory (cf \cite{KPT}, section 2), reducts and expansions (cf \cite{KPT}, section 5). If $L$ is a first order signature and $\m A$ and $\m B$ are $L$-structures, we will write $\m A\leq \m B$ when $\m A$ embeds in $\m B$, $\m A \subset \m B$ when $\m A$ is a substructure of $\m B$, and $\m A \cong \m B$ when $\m A$ and $\m B$ are isomorphic. If $C$ is a subset of the universe of $\m A$ which supports a substructure of $\m A$, we will write $\restrict{\m A}{C}$ for the corresponding substructure.

A \emph{Fra\"iss\'e class} in a countable first order language $L$ will be a countable class of finite $L$-structures of arbitrarily large sizes, satisfying the hereditarity, joint embedding and amalgamation property, and a \emph{Fra\"iss\'e structure} (or \emph{Fra\"iss\'e limit}) in $L$ will be a countable, locally finite, ultrahomogeneous $L$-structure. In \cite{KPT}, two combinatorial properties of classes of finite structures have a considerable importance. Those are the \emph{Ramsey property} and the \emph{ordering property}. 

In order to define the Ramsey property, let $k, l\in \N$, and $\m A, \m B, \m C$ be $L$-structures. The set of all \emph{copies} of $\m A$ in $\m B$ is written using the binomial notation $$ \binom{\m B}{\m A} = \{ \mc A \subset \m B : \mc A \cong \m A\}.$$ 

We use the standard arrow partition symbol $$ \m C \arrows{(\m B)}{\m A}{k, l}$$ to mean that for every map $c: \funct{\binom{\m C}{\m A}}{[k]}$, thought as a $k$-coloring of the copies of $\m A$ in $\m C$, there is $\mc B \in \binom{\m C}{\m B}$ such that $c$ takes at most $l$-many values on $\binom{\mc B}{\m A}$. When $l=1$, this is written  $$ \m C \arrows{(\m B)}{\m A}{k}.$$ 

A class $\mathcal{K}$ of finite $L$-structures is then said to have the \emph{Ramsey property} when $$ \forall k \in \N \quad \forall \m A, \m B \in \mathcal{K} \quad \exists \m C \in \mathcal{K} \quad \m C \arrows{(\m B)}{\m A}{k}.$$

When $\mathcal{K} = \Age (\m F)$, where $\m F$ is a Fra\"iss\'e structure, this is equivalent, via a compactness argument (detailed in Proposition \ref{prop:RPcompact}), to: $$ \forall k \in \N \quad \forall \m A, \m B \in \mathcal{K} \quad \m F \arrows{(\m B)}{\m A}{k}.$$

As for the ordering property, assume that $<$ is a binary relation symbol not contained in $L$, and that $L^{*}=L\cup\{<\}$. Let $\mathcal{K}$ be a Fra\"iss\'e class in $L$, and $\mathcal{K}^{*}$ an \emph{order expansion} of $\mathcal K$ in $L^{*}$. That means that all elements of $\mathcal{K}^{*}$ are of the form $\m A ^{*} = (\m A, <^{\m A})$, where $\m A \in \mathcal{K}$ and $<^{\m A}$ is a linear ordering on the universe $A$ of $\m A$ ($\m A$ is then the \emph{reduct} of $\m A^{*}$ to $L$ and is denoted $ \restrict{\m A^{*}}{L}$), and that, conversely, any $\m A \in \mathcal K$ admits a linear ordering $<^{\m A}$ so that $(\m A, <^{\m A}) \in \Age(\m F^{*})$. Then, $\mathcal K ^{*}$ satisfies the \emph{ordering property} relative to $\mathcal K$ if, for every $ \m A \in \mathcal{K}$, there exists $ \m B \in \mathcal{K}$ such that $$\forall \m A^{*}, \m B^{*} \in\mathcal{K} ^{*} \quad (\restrict{\m A^{*}}{L} = \m A \ \  \wedge \ \ \restrict{\m B^{*}}{L} = \m B) \Rightarrow \m A^{*} \leq \m B^{*}. $$

Note that previously, the restriction symbol was also used to refer to substructures as opposed to reducts. Because the context almost always prevents the confusion between those two notations, we will use freely both of them, without any further indication.

We now turn to topological groups and to dynamical properties of their actions. Let $G$ be a topological group. A \emph{$G$-flow} is a continuous action of $G$ on a topological space $X$. We will often use the notation $G \curvearrowright X$. The flow $G \curvearrowright X$ is \emph{compact} when the space $X$ is. It is \emph{minimal} when every $x \in X$ has dense orbit in $X$: \[ \forall x \in X \  \  \overline{G\cdot x} = X\] 

Finally, it is \emph{universal} when every compact minimal $G \curvearrowright Y$ is a factor of $G \curvearrowright X$, which means that there exists $\pi : X \longrightarrow Y$ continuous, onto, and so that $$\forall g \in G \quad \forall x \in X \quad  \pi (g \cdot x) = g \cdot \pi(x).$$ 

It turns out that when $G$ is Hausdorff, there is, up to isomorphism of $G$-flows, a unique $G$-flow that is both minimal and universal. This flow is called \emph{the universal minimal flow} of $G$, and is denoted $G \curvearrowright M(G)$. When the space $M(G)$ is reduced to a singleton, the group $G$ is said to be \emph{extremely amenable}. Equivalently, every compact $G$-flow $G\curvearrowright$ admits a fixed point, i.e. an element $x\in X$ so that $g\cdot x = x$ for every $g\in G$. We refer to \cite{KPT} or \cite{Pe} for a detailed account on those topics. Let us simply mention that, concerning extreme amenability, it took a long time before even proving that such groups exist, but that several non-locally compact transformation groups are now known to be extremely amenable (the most remarkable ones being probably the isometry groups of the separable infinite dimensional Hilbert space (Gromov-Milman, \cite{GM}), and of the Urysohn space (Pestov, \cite{Pe0})). As for universal minimal flows, prior to \cite{KPT}, only a few cases were known to be both metrizable and non-trivial, the most important examples being provided by the orientation-preserving homeomorphisms of the circle (Pestov, \cite{Pe1}), $S_{\infty}$ (Glasner-Weiss, \cite{GW1}), and the homeomorphism group of the Cantor space (Glasner-Weiss, \cite{GW2}). In that context, the paper \cite{KPT} established a link between Ramsey property and extreme amenability. For an $L$-structure $\m A$, we denote by $\Aut (\m A)$ the corresponding automorphism group. When this group is trivial, we say that $\m A$ is \emph{rigid}.

\begin{thm}[Kechris-Pestov-Todorcevic, \cite{KPT}, essentially Theorem 4.8]

\label{thm:EARP}

Let $\m F$ be a Fra\"{i}ss\'e structure, and let $G = \Aut(\m F)$. The following are equivalent: 

\begin{enumerate}
\item[i)] The group $G$ is extremely amenable. 

\item[ii)] The class $\Age (\m F)$ has the Ramsey property and consists of rigid elements.  
\end{enumerate}

\end{thm}

Because closed subgroups of $S_{\infty}$ are all of the form $\Aut(\m F)$, where $\m F$ is a Fra\"{i}ss\'e structure, the previous theorem actually completely characterizes those closed subgroups of $S_{\infty}$ that are extremely amenable. It also allows the description of many universal minimal flows via combinatorial methods. Indeed, when $\m F^{*} = (\m F, <^{*})$ is an order expansion of $\m F$, one can consider the space $\LO(\m F)$ of all linear orderings on $\m F$, seen as a subspace of $[2]^{\m F\times \m F}$. In this notation, the factor $[2]^{\m{F}\times \m F} = \{ 0, 1\}^{\m{F}\times \m F}$ is thought as the set of all binary relations on $\m F$. This latter space is compact, and $G$ continuously acts on it: if $S\in [2]^{\m{F}\times \m F}$ and $g\in G$, then $g\cdot S$ is defined by $$\forall x, y \in \m F \quad g\cdot S(x, y) \Leftrightarrow S(g^{-1}(x), g^{-1}(y)).$$ 

It can easily be seen that $\LO (\m F)$ and $X^{*} : = \overline{G \cdot <^{*}}$ are closed $G$-invariant subspaces.

\begin{thm}[Kechris-Pestov-Todorcevic, \cite{KPT}, Theorem 7.4]

\label{thm:OP}

Let $\m F$ be a Fra\" iss\'e structure, and $\m F^{*}$ a Fra\"iss\'e order expansion of $\m F$. The following are equivalent: 

\begin{enumerate}

\item[i)] The flow $G\curvearrowright X^{*}$ is minimal. 

\item[ii)] $\Age(\m F^{*})$ has the ordering property relative to $\Age(\m F)$. 

\end{enumerate}

\end{thm}

The following result, which builds on the two preceeding theorems, is then obtained: 

\begin{thm}[Kechris-Pestov-Todorcevic, \cite{KPT}, Theorem 10.8]

\label{thm:KPTUMF}

Let $\m F$ be a Fra\" iss\'e structure, and $\m F^{*}$ be a Fra\"iss\'e order expansion of $\m F$. The following are equivalent: 

\begin{enumerate}

\item[i)] The flow $G\curvearrowright X^{*}$ is the universal minimal flow of $G$. 

\item[ii)] The class $\Age(\m F^{*})$ has the Ramsey property as well as the ordering property relative to $\Age(\m F)$. 

\end{enumerate}

\end{thm}

A direct application of those results allowed to find a wealth of extremely amenable groups and of universal minimal flows, see (\cite{KPT}, Sections 6 and 8), but also \cite{Ne}, \cite{NVT1}, \cite{So1}, \cite{So2} and \cite{J}. However, some cases, which are very close to those described above, \emph{cannot} be captured directly by those theorems. Precisely, some Fra\"iss\'e classes do not have an order expansion with the Ramsey and the ordering property, but do so when the language is enriched with  additional symbols. Some examples already appear in \cite{KPT} (e.g. Theorem 8.4 dealing with equivalence relations with the number of classes bounded by a fixed number).  It is also the case for equivalence relations whose classes have a size bounded by a fixed number, for the subtournaments of the dense local order (See \cite{LNS}, or Section \ref{section:Cy} of the present paper), as well as for several classes of finite posets (see \cite{So2}). More recently, Jasi\'nski showed that boron tree structures have the same property, see \cite{J}. For all those cases, a slight modification of the original framework does allow to describe the universal minimal flow. The purpose of the present paper is to make this method explicit and to illustrate how it can be applied in concrete situations. 

Precisely, we will not deal with order expansions in the language $L^{*}=L\cup \{<\}$ only (those will be later on referred to as \emph{pure} order expansions), but with \emph{precompact relational expansions}. For such expansions, we do not require $L^{*}=L\cup \{<\}$, but only $L^{*} = L\cup \{ R_{i}:i\in I\}$, where $I$ is countable, and every symbol $R_{i}$ is relational and not in $L$. An expansion $\m F^{*}$ of $\m F$ is then called \emph{precompact} when any $\m A \in \Age(\m F)$ only has finitely many expansions in $\Age(\m F^{*})$. Note that every $\m A \in \Age(\m F)$ has at least one expansion in $\Age(\m F ^{*})$: simply take a copy of $\m A$ in $\m F$, and consider the substructure of $\m F^{*}$ that it supports. The choice of the terminology is justified in Section \ref{section:precompact}. For those expansions, the ordering property has a direct translation, which we call the \emph{expansion property}, and Theorems \ref{thm:OP} and \ref{thm:KPTUMF} turn into the following versions: 

\begin{thm}

\label{thm:EP}

Let $\m F$ be a Fra\" iss\'e structure, and $\m F^{*}$ a precompact relational expansion of $\m F$ (not necessarily Fra\"iss\'e). The following are equivalent: 

\begin{enumerate}

\item[i)] The flow $G\curvearrowright X^{*}$ is minimal. 

\item[ii)] $\Age(\m F^{*})$ has the expansion property relative to $\Age(\m F)$. 

\end{enumerate}

\end{thm}

\begin{thm}

\label{thm:UMF}

Let $\m F$ be a Fra\" iss\'e structure, and $\m F^{*}$ be a Fra\"iss\'e precompact relational expansion of $\m F$. Assume that $\Age(\m F^{*})$ consists of rigid elements. The following are equivalent: 

\begin{enumerate}

\item[i)] The flow $G\curvearrowright X^{*}$ is the universal minimal flow of $G$. 

\item[ii)] The class $\Age(\m F^{*})$ has the Ramsey property as well as the expansion property relative to $\Age(\m F)$. 

\end{enumerate}

\end{thm}

Another common point with pure order expansions is the following, purely combinatorial, result, which can be thought as the precompact version of Theorem 10.7 from \cite{KPT}.

\begin{thm}

\label{thm:subclass}

Let $\m F$ be a Fra\" iss\'e structure, and $\m F^{*}$ be a Fra\"iss\'e precompact relational expansion of $\m F$. Assume that $\Age(\m F^{*})$ consists of rigid elements, and has the Ramsey property. Then $\Age(\m F^{*})$ admits a Fra\"iss\'e subclass with the Ramsey property and the expansion property relative to $\Age(\m F)$.  

\end{thm}

Those results are proved in Sections \ref{section:min}, \ref{section:UMF} and \ref{section:subclass} respectively. To illustrate their use, the universal minimal flows of the automorphism groups of the circular directed graphs $\Cy$ and $\Cz$ are computed in Sections \ref{section:Cy} and \ref{section:Cz}. 

Two aspects should be emphasized. First, the only reason for which the original paper \cite{KPT} was not written in the general setting we present here is that, at the time where it was developed, pure order expansions covered almost all known applications of the method to compute universal minimal flows (the cases that were left aside were computed easily with a bit of extra work). Arguably, they consequently constituted the right setting to establish a general correspondence. Interestingly, the fact that the general picture is actually a bit bigger could be an opportunity to change the way we think of structural Ramsey theorems for Fra\"iss\'e classes. Indeed, when analyzing how the most famous results of the field were obtained, it seems that two categories emerge. The first one corresponds to those ``natural'' classes where the Ramsey property holds: finite sets, finite Boolean algebras, finite vector spaces over a finite field. The second one corresponds to those classes where the Ramsey property fails but where this failure can be fixed by adjoining a linear ordering: finite graphs, finite $K_{n}$-free graphs, finite hypergraphs, finite partial orders, finite topological spaces, finite metric spaces. As for those classes where more than a linear ordering is necessary, we have to admit that besides the ones that appear in \cite{KPT} (finite equivalence relations with classes of size bounded by $n$, or equivalence relations with at most $n$ classes) or those, more recent, that we mentioned previously (namely, subtournaments of $\Cy$, subtournaments of $\Cz$, posets that are unions of at most $n$ many chains, posets that are obtained as a totally ordered set of antichains of size at most $n$, and boron tree structures), we are not aware of any additional case, but it would be extremely surprising that nobody encountered such instance before. More likely is the fact that the corresponding results were not considered as true structural Ramsey results, and were therefore overlooked. However, in our opinion, the results of the present paper seem to give the hint that some valuable material may well be found there. For example, they allow to compute the universal minimal flow for \emph{every} automorphism group coming from countable ultrahomogeneous graphs, posets, and tournaments. In fact, we are not aware of any example of a countable ultrahomogeneous structure in a finite language (or, more generally, of a countable ultrahomogeneous $\omega$-categorical structure) where this is not so. Of course, in order to see wether this is a general phenomenon, a natural thing is to examine Cherlin's class of directed graphs. This analysis will be carried out in a forthcoming paper. More on the relevance of precompact expansions is included in Section \ref{section:why}.

Second, as far as the proofs are concerned, it would have been possible to keep all the original arguments. We chose not to completely do so, and to take advantage of the opportunity to make a slightly different, concise and self-contained, exposition, somewhere between \cite{KPT}  and \cite{Pe}. This choice explains why the proofs of Theorems \ref{thm:EP}, \ref{thm:UMF}, \ref{thm:subclass} and even \ref{thm:EARP} are included, when the novelty really concerns the description of the universal minimal flows $\Aut(\Cy)$ and $\Aut(\Cz)$ and takes place in Sections \ref{section:Cy} and \ref{section:Cz}. Of course, for a detailed exposition of the Kechris-Pestov-Todorcevic correspondence, the reader is urged to consult the original article \cite{KPT}, which contains far more than what we chose to cover in the present paper. 

\

\textbf{Acknowledgements}: This project was initiated in 2007, while I was a postdoctoral fellow at the University of Calgary under the supervision of Claude Laflamme and Norbert Sauer. I would therefore like to thank both of them, and to acknowledge support from the Department of Mathematics \& Statistics Postdoctoral Program at the University of Calgary in 2007 and in 2008. I would also like to thank Alexander Kechris, Vladimir Pestov, Miodrag Soki\'c, Stevo Todorcevic, Todor Tsankov and the anonymous referee for helpful comments and suggestions. 

\section{Precompact relational expansions}

\label{section:precompact}

In what follows, $\m{F}$ is a Fra\"{i}ss\'e structure in some countable language $L$, $L^{*}$ is an expansion of $L$ such that $L^{*}\smallsetminus L = \{ R_{i} : i\in I\}$ is countable and consists only of relation symbols. For $i\in I$, the arity of the symbol $R_{i}$ is denoted $a(i)$. $\m{F}^{*}$ is an expansion of $\m{F}$ in $L^{*}$, which does not have to be Fra\"iss\'e at the moment.  We write $\m{F}^{*} = (\m{F}, (R^{*}_{i})_{i\in I})$, or $(\m{F}, \vec{R}^{*})$. We also assume that $\m F$ and $\m F^{*}$ have the set $\N$ of natural numbers as universe. 

The corresponding automorphism groups are denoted $G$ and $G^{*}$ respectively. The group $G^{*}$ will be thought as a subgroup of $G$, and both are closed subgroups of $S_{\infty}$, the permutation group of $\N$ equipped with the topology generated by sets of the form $$ U_{g, F} = \{ h\in G : \restrict{h}{F} = \restrict{g}{F}\},$$ where $g$ runs over $G$ and $F$ runs over all finite subsets of $\N$. This topology admits two natural uniform structures, a left-invariant one, $\mathcal U_{L}$, whose basic entourages are of the form $$ U^{L}_{F} = \{ (g,h) : g^{-1}h \in U_{e, F} \}, \ \ F\subset \N \ \ \textrm{finite},$$ and a right-invariant one, $\mathcal U _{R}$, whose basic entourages are of the form $$ U^{R} _{F} = \{ (g,h) : (g^{-1}, h^{-1}) \in U^{L} _{F}\},\ \ F\subset \N \ \ \textrm{finite}.$$

In fact, those two uniform structures are respectively generated by the two following ultrametrics: $d_{L}$, defined as $$ d_{L}(g,h) = \frac{1}{2^{m}}, \quad m=\min \{ n\in \N : g(n)\neq h(n)\},$$ and $d_{R}$, given by $$d_{R}(g,h) = d_{L}(g^{-1}, h^{-1}).$$

In what follows, we will be interested in the set of all expansions of $\m F$ in $L^{*}$, which we think as the product $$P^{*} := \prod_{i\in I} [2]^{\m{F}^{a(i)}}.$$ 

In this notation, the factor $[2]^{\m{F}^{a(i)}} = \{ 0, 1\}^{\m{F}^{a(i)}}$ is thought as the set of all $a(i)$-ary relations on $\m F$. Each factor $[2]^{\m{F}^{a(i)}}$ is equipped with an ultrametric $d_{i}$, defined by $$ d_{i}(S_{i},T_{i}) = \frac{1}{2^{m}}, \quad m = \min\{ n\in \N : \restrict{S_{i}}{[n]} \neq \restrict{T_{i}}{[n]}\}$$ where $\restrict{S_{i}}{[m]}$ (resp. $\restrict{T_{i}}{[m]}$) stands for $S_{i} \cap [m]^{a(i)}$ (resp. $T_{i} \cap [m]^{a(i)}$). So $\restrict{S_{i}}{[m]} = \restrict{T_{i}}{[m]}$ means that $$\forall y_{1}\ldots y_{a(i)} \in [m] \quad S_{i}(y_{1}\ldots y_{a(i)}) \Leftrightarrow T_{i}(y_{1}\ldots y_{a(i)}).$$

The group $G$ acts continuously on each factor as follows: if $i\in I$, $S_{i}\in [2]^{\m{F}^{a(i)}}$ and $g\in G$, then $g\cdot S_{i}$ is defined by $$\forall y_{1}\ldots y_{a(i)} \in \m F \quad g\cdot S_{i}(y_{1}\ldots y_{a(i)}) \Leftrightarrow S_{i}(g^{-1}(y_{1})\ldots g^{-1}( y_{a(i)})).$$

This allows to define an action of $G$ on the product $P^*$, where $g\cdot \vec S$ is simply defined as $(g\cdot S_{i})_{i\in I}$ whenever $\vec S = (S_{i})_{i\in I} \in P^*$ and $g\in G$. 

This action is continuous when $P^*$ is equipped with the product topology (it is then usually referred to as the \emph{logic action}), but also when it is endowed with the supremum distance $d^{P^*}$ of all the distances $d_{i}$. The corresponding topology is finer than the product topology if $I$ is infinite, but it is the one we will be interested in in the sequel because of its connection to the quotient $G/G^{*}$.

\begin{prop}

\label{prop:finite}

The metric subspace $G\cdot \vec R^{*}\subset P^{*}$ is precompact iff every element of $\Age(\m F)$ has finitely many expansions in $\Age(\m F^{*})$. 

\end{prop}

Recall that a metric space $X$ is precompact when its completion is compact. Equivalently, it can be covered by finitely many balls of arbitrary small diameter. When the space is only uniform as opposed to metric, that means that for every basic entourage $V$, there are finitely many $x_{1}, \ldots, x_{n}$ so that the family of sets $(\{ x \in X : (x, x_{i})\in V\})_{i\leq n}$ covers $X$.

\begin{proof}
Observe first that by ultrahomogeneity of $\m F$, every expansion $\m A^{*}$ of any finite substructure $\m A \subset \m F$ can be realized in the following sense: $$\exists g \in G \quad \restrict{(\m F, g\cdot\vec R^{*})}{A} \cong \m A^{*}$$

Suppose that $G\cdot \vec R^{*}$ is precompact. Fix $m$ large enough so that $A \subset [m]$. It is possible to cover $G\cdot \vec R^{*}$ by finitely many balls of radius $1/2^{m}$, call them $B_{1}, \ldots, B_{l}$. Note that if $\vec S, \vec T$ belong to the same ball, then $\restrict{\vec S}{[m]} = \restrict{\vec T}{[m]}$. In particular, $\restrict{\vec S}{A} = \restrict{\vec T}{A}$. It follows that there are at most $l$-many non-isomorphic structures of the form $\restrict{(\m F, g\cdot \vec R^{*})}{A}$, and therefore that $\m A$ has at most $l$-many expansions in $\Age(\m F^{*})$. 

Conversely, suppose that every element of $\Age(\m F)$ has finitely many expansions in $\Age(\m F^{*})$. Let $m \in \N$. We are going to cover $G\cdot \vec R^{*}$ with finitely many balls of radius at most $1/2^{m}$. Call $\m A$ the substructure of $\m F$ generated by $[m]$, and let $\m{A}^{*}_{1},\ldots, \m{A}^{*}_{l}$ denote all the expansions of $\m{A}$ in $\Age(\m F)$. For $j\leq l$, let $$B_{j} = \{ \vec S \in G\cdot\vec R^{*} : \restrict{(\m{F}, \vec S)}{A}\cong \m A^{*}_{j}\}.$$ 

Then $B_{1}, \ldots, B_{l}$ are balls of radius at most $1/2^{m}$ and because $\m{A}^{*}_{1},\ldots, \m{A}^{*}_{l}$ exhaust all expansions of $\m{A}$ in $\Age(\m F^{*})$, we also have $$ G\cdot \vec R^{*} = \bigcup _{j=1}^{l}B_{j}. \qquad \qedhere$$ \end{proof}

\begin{defn}

\label{defn:precompact}

The expansion $\m F^{*} = (\m F , \vec R^{*})$ is a \emph{precompact} relational expansion of $\m F$ when every element of $\Age (\m F)$ only has finitely many expansions in $\Age (\m F^{*})$. Equivalently, the metric subspace $G\cdot \vec R^{*}$ of $P^{*}$ is precompact. In that case, we denote by $X^{*}$ the corresponding completion, ie $$X ^{*} = \overline{G\cdot \vec R^{*}}\quad (\textrm{where the closure in taken in $P^*$, which is complete}).$$

\end{defn}

We will see now that when $\m F^{*}$ is Fra\"iss\'e, there is a close connection between the metric space $(P^{*}, d^{P^{*}})$ with the quotient $G/G^{*}$. As a set, $G/G^{*}$ can be thought as $G\cdot\vec{R}^{*}$, the orbit of $\vec{R}^{*}$ in $P^*$, by identifying $[g]$, the equivalence class of $g$,  with $g\cdot \vec{R}^{*}$ (recall that $\vec R^{*}$ is defined as $\m F^{*} = (\m F, \vec R^{*})$). With this identification, the logic action of $G$ on $G\cdot\vec{R}^{*}$ coincides with the natural action on $G/G^{*}$ by left translations. The two uniform structures on $G$ project onto uniform structures on $G/G^{*}$, but we will pay a particular attention to the projection of $\mathcal U _{R}$, whose basic entourages are of the form $$V_{F} = \{ ([g], [h]) : \restrict{g^{-1}}{F} = \restrict{h^{-1}}{F} \},\ \ F\subset \N \ \ \textrm{finite}.$$

% $$ \rho_{R}([g],[h]) = $$

\begin{prop}

\label{prop:unif}

If $\m F^{*}$ is Fra\"iss\'e, then the projection of $\mathcal U_{R}$ on $G/G^{*}\cong G\cdot\vec{R}^{*}$ coincides with the uniform structure induced by the restriction of $d^{P^*}$ on $G\cdot\vec{R}^{*}$.   

\end{prop}

%Note that this result does \emph{not} hold for the left uniform structure on $G/G^{*}$, see \cite{Pe} p.128 for more on that aspect.  

\begin{proof}

The distance $d^{P^*}$ induces a uniform structure whose entourages are generated by sets of the form $$W_{m} = \{ (\vec S, \vec T) : d^{P^*}(\vec S, \vec T)< \frac{1}{2^{m}}\} = \{ (\vec S, \vec T) : \restrict{\vec S}{[m]} = \restrict{\vec T}{[m]}\},$$ where $\restrict{\vec S}{[m]} = \restrict{\vec T}{[m]}$ means that $$\forall i \in I \quad \restrict{S_{i}}{[m]} = \restrict{T_{i}}{[m]}.$$

Equivalently, because $\m F$ is locally finite, this uniform structure is generated by those sets of the form $$W_{F} = \{ (\vec S, \vec T) : \restrict{\vec S}{F} = \restrict{\vec T}{F}\},$$ where $F$ is a finite subset of $\N$ supporting a substructure of $\m F$. Let $F$ be such a finite set. From the definition of $V_{F}$ and $W_{F}$, it is clear that $V_{F}\subset (W_F\cap G/G^{*})$. We are going to show that $(W_{F}\cap G/G^{*}) \subset V_{F}$ also holds. Let $(\vec S, \vec T) \in W_{F}\cap G/G^{*} $, and  fix $s, t \in G$ so that $\vec S = s \cdot \vec R^{*}$ and $\vec T = t\cdot \vec R^{*}$. Then, for every $ i \in I $ and every $ y_{1}\ldots y_{a(i)} \in F$, $$R^{*}_{i}(s^{-1}(y_{1})\ldots s^{-1}(y_{a(i)})) \Leftrightarrow R^{*}_{i}(t^{-1}(y_{1})\ldots t^{-1}(y_{a(i)})).$$ 

Therefore, the map $s^{-1}(F)\longrightarrow t^{-1}(F)$ defined by $s^{-1}(n)\longmapsto t^{-1}(n)$ for every $n\in F$ is an isomorphism between $s^{-1}(F)$ and $t^{-1}(F)$ seen as substructures of $\m F^{*}$. By ultrahomogeneity of $\m F^{*}$ (this is where we use that $\m F^{*}$ is Fra\"iss\'e), it is possible to extend this isomorphism to some element $g\in G^{*}$. We then have: $$\forall n \in F \quad gs^{-1} (n) = t^{-1} (n).$$ 

Since $gs^{-1} = (sg^{-1})^{-1}$, we obtain $([sg^{-1}],[t]) \in V_{F}$. But $[sg^{-1}] = [s]$, so $([s],[t]) \in V_{F}$, ie $(\vec S, \vec T) \in V_{F}$.  \qedhere

\end{proof}

Therefore, in the sequel, when $\m F^{*}$ is Fra\"iss\'e, we will really think of the uniform space $G/G^{*}$ as the metric subspace $G\cdot \vec R^{*}$ of $P^*$. 

\begin{cor}

Assume that $\m F^{*}$ is Fra\"iss\'e. Then $\m F^{*}$ is a precompact expansion of $\m F$ iff the uniform space $G/G^{*}$ is precompact. In that case, we can identify the compact spaces $X^{*}$ and $\widehat{G/G^{*}}$ as well as the flows $G \curvearrowright X^{*}$ and $G\curvearrowright\widehat{G/G^{*}}$. 

\end{cor}

In what follows, we will be interested in generalizing the theory of Kechris-Pestov-Todorcevic to precompact relational expansions instead of pure order expansions. However, before doing so, let us provide a few examples. Taking $\m F = \N$ (the language $L$ is then empty), we have $G=S_{\infty}$ and a Fra\"iss\'e expansion $\m F^{*}$ of $\m F$ such that $L^{*}\smallsetminus L$ is relational is simply a relational Fra\"iss\'e structure. It is a precompact expansion of $\N$ exactly when $\m F^{*}$ only has finitely many substructures up to isomorphism in each finite cardinality. The group $G^{*}$ is then called \emph{oligomorphic}. A classical case where this happens is when the language $L^{*}$ is relational and only has finitely many symbols in each arity (e.g. graphs, directed graphs,...). On the other hand, there are also natural Fra\"iss\'e relational structures which are not precompact expansions of $\N$. Countable ultrahomogeneous metric spaces with infinitely many distances fall into that category. The appropriate language is made of binary relations symbols $(R_{\alpha})_{\alpha}$ and $R_{\alpha}(x,y)$ holds exactly when $d(x,y) = \alpha$. One typical example is the \emph{rational Urysohn space} $\textbf{U} _{\Q}$, which is, up to isometry, the unique countable ultrahomogeneous metric space with rational distances into which any finite metric space with rational distances embeds.

\section{Extreme amenability and Ramsey property}

The purpose of this section is to provide a proof of Theorem \ref{thm:EARP} (again, this is only done for the sake of completeness):

\begin{thmm}[Kechris-Pestov-Todorcevic, \cite{KPT}, Theorem 4.8]

Let $\m F$ be a Fra\"{i}ss\'e structure, and let $G = \Aut(\m F)$. The following are equivalent: 

\begin{enumerate}
\item[i)] The group $G$ is extremely amenable. 

\item[ii)] The class $\Age (\m F)$ has the Ramsey property and consists of rigid elements.  
\end{enumerate}

\end{thmm}

We first detail the argument according to which the Ramsey property for $\Age(\m F)$ is equivalent to its infinite version with $\m F$. 

\begin{prop}

\label{prop:RPcompact}

The class $\Age (\m F)$ has the Ramsey property iff $$ \forall k \in \N \quad \forall \m A, \m B \in \mathcal{K} \quad \m F \arrows{(\m B)}{\m A}{k}.$$

\end{prop}

\begin{proof}
Because every $\m C \in \Age(\m F)$ embeds in $\m F$, it is clear that the Ramsey property implies $$ \forall k \in \N \quad \forall \m A, \m B \in \mathcal{K} \quad \m F \arrows{(\m B)}{\m A}{k}.$$

Conversely, suppose that the Ramsey property does not hold. Then there exist $k\in \N$, $\m A, \m B \in \Age(\m F)$ so that no $\m C \in \Age(\m F)$ satisfies $$\m C \arrows{(\m B)}{\m A}{k}.$$ 

Equivalently, for every $\m C\in \Age(\m F)$, there exists a coloring $ \chi_{\m C} : \funct{\binom{\m C}{\m A}}{[k]}$ with no monochromatic set of the form $\binom{\mc B}{\m A}$. Consider a free ultrafilter $U$ on the set of finite non empty subsets of $\N$ so that for every finite $D \subset \N$, $$\{ E\subset \N : D\subset E \ \wedge \ \restrict{\m F}{E} \in \Age(\m F)\}\in U.$$  

Note that the local finiteness of $\m F$ indeed guarantees that such an ultrafilter exists. For $\mc A \in \binom{\m F}{\m A}$ and $\varepsilon \in [k]$, define $$ K^{\varepsilon} _{\mc A} = \{ E\subset \N : \mc A \subset (\restrict{\m F}{E}) \  \wedge  \ \chi_{\restrict{\m F}{E}}(\mc A) = \varepsilon\}.$$

Then, define a coloring $\chi : \funct{\binom{\m F}{\m A}}{[k]}$ by $$\chi(\mc A) = \varepsilon \quad \Leftrightarrow \quad K^{\varepsilon} _{\mc A} \in U.$$ 

Let $\mc B$ be an arbitrary copy of $\m B$ in $\m F$. We claim that $\binom{\mc B}{\m A}$ is not $\chi$-monochromatic. Towards a contradiction, assume the contrary, and call $\varepsilon_{0}$ the corresponding constant value of $\chi$. Then the following set is in $U$: $$\{ E\subset \N : \mc B\subset (\restrict{\m F}{E})\} \cap \bigcap_{\mc A \in \binom{\mc B}{\m A}} K^{\varepsilon_{0}} _{\mc A}$$

Therefore, it is not empty and contains some finite set $E_{0}$. Then $\m B\subset (\restrict{\m F}{E_{0}})$ and $\binom{\restrict{\m F}{E_{0}}}{\m A}$ is $\chi_{\restrict{\m F}{E_{0}}}$-monochromatic with color $\varepsilon_{0}$. Hence, $\binom{\mc B}{\m A}$ is $\chi_{\restrict{\m F}{E_{0}}}$-monochromatic, a contradiction. \qedhere

\end{proof}

We now turn to a proof of Theorem \ref{thm:EARP}. 

\begin{proof}[Proof of $i) \Rightarrow ii)$]

Assume that $G$ is extremely amenable. We first prove that elements of $\Age (\m F)$ are rigid. To do so, consider the set of all linear orderings $LO(\m F)$, seen as a subspace of the space $[2]^{\m F \times \m F}$. The group $G$ acts continuously on this later space via the logic action. The set $LO(\m F)$ is then a $G$-invariant compact subspace. Explicitly, $G$ acts on $LO(\m F)$ as follows: if $ \prec \in LO(\m F)$ and  $g \in G$, then $$ \forall x, y \in \m F \quad x (g\cdot\prec) y \Leftrightarrow g^{-1} (x) \prec g^{-1} (y).$$

By extreme amenability of $G$, there is a $G$-fixed point in $LO(\m F)$, call it $<$. Consider now a finite substructure $\m A \subset \m F$, and let $\varphi$ be an automorphism of $\m A$. By ultrahomogeneity of $\m F$, $\varphi$ extends to an automorphism $\phi$ of $\m F$. Because $<$ is $G$-fixed, it is preserved under $\phi$. Thus, on $A$, $<$ is preserved by $\varphi$, which means that $\varphi$ is trivial on $A$. This proves that $\m A$ is rigid.   

To prove that $\Age (\m F)$ has the Ramsey property, consider $k\in \N$, $\m A$ and $\m B$ in $\Age(\m F)$, and a coloring $$ c : \funct{\binom{\m F}{\m A}}{[k]}.$$ 

Consider the compact space $[k]^{\binom{\m F}{\m A}}$, acted on continuously by $G$ by shift: if $\chi \in [k]^{\binom{\m F}{\m A}}$, $g\in G$ and $\mc A \in \binom{\m F}{\m A}$, then $$ g\cdot \chi (\mc A) = \chi(g^{-1} (\mc A)).$$

The set $\overline{G\cdot c}$ is a $G$-invariant compact subspace. By extreme amenability of $G$, there is a $G$-fixed point in $\overline{G\cdot c}$, call it $c_{0}$. The fact that $c_{0}$ is $G$-fixed means that $c_{0}$ is constant. Consider now the finite set $\binom{\m B}{\m A}$. Because $c_{0} \in \overline{G\cdot c}$, there is $g \in G$ so that $$\restrict{g\cdot c}{\binom{\m B}{\m A}} = \restrict{c_{0}}{\binom{\m B}{\m A}}.$$ 

So $g\cdot c$ is constant on $\binom{\m B}{\m A}$, and $c$ is constant on $\binom{g^{-1}(\m B)}{\m A}$. Because $g^{-1}(\m B)$ is isomorphic to $\m B$, this proves that $\Age (\m F)$ has the Ramsey property.  \qedhere

\end{proof}

\begin{proof}[Proof of $ii)\Rightarrow i)$]

Assume that $\Age (\m F)$ has the Ramsey property and consists of rigid elements. For $A \subset \N$ finite and supporting a substructure $\m A \subset \m F$, the rigidity of $\m A$ implies that the setwise stabilizer of $A$ in $G$ is equal to the pointwise stabilizer $\Stab(A)$ in $G$, and we can identify $G/\Stab (A)$ with a subset of $\binom{\m F}{\m A}$. Moreover, because $\m F$ is ultrahomogeneous, we also have the reverse inclusion. Thus, we can make the identification: $$G/\Stab (A) = \binom{\m F}{\m A}.$$

\begin{prop}

\label{prop:RP0}

Let $k\in \N$, $A\subset \N$ finite and supporting a substructure $\m A$ of $\m F$, and $F\subset G$ finite. Let $\bar f : \funct{G}{[k]}$, constant on elements of $G/\Stab (A) $. Then there exists $g\in G$ such that $\bar f$ is constant on $gF$. 

\end{prop}

\begin{proof}

The map $\bar f$ induces a map $f : \funct{G/\Stab (A)}{[k]}$, which we may think as a $k$-coloring of $\binom{\m F}{\m A}$. Consider the set $ \{ [h] : h \in F\}$. It is a finite set of substructures of $\m F$, all isomorphic to $\m A$. Therefore, we can find a finite substructure $\m B\subset \m F$ large enough so that $$ \{ [h] : h \in F\} \subset \binom{\m B}{\m A}.$$

By Ramsey property, find $\mc B \in \binom{\m F}{\m B}$ so that $f$ is constant on $\binom{\mc B}{\m A}$, with value $i<k$. By ultrahomogeneity of $\m F$, find $g\in G$ so that $g(\m B) = \mc B$. We claim that $g$ is as required. Indeed, for $h \in F$, we have $$ [gh] = g([h]) \in \binom{g(\m B)}{\m A} = \binom{\mc B}{\m A}.$$

So $$ \bar f (gh) = f([gh]) = f(g[h]) = i. \qquad \qedhere$$ \end{proof}

\begin{prop}

\label{prop:RP1}

Let $p\in \N$, $f : \funct{G}{\R^{p}}$ left uniformly continuous and bounded (where $\R^{p}$ is equipped with its standard Euclidean structure), $F\subset G$ finite, $\varepsilon > 0$. Then there exists $g \in G$ such that $$ \forall h, h' \in F \quad \| f(gh) - f(gh')\|<\varepsilon.$$ 

\end{prop}

\begin{proof}

Let $m\in \N$. Note that as subsets of $G$, elements of $G/\Stab ([m]) $ have diameter $1/2^{m+1}$ with respect to the left invariant metric $d_{L}$ on $G$. Thus, by left uniform continuity of $f$, we can find $m\in \N$ large enough so that $f$ is constant up to $\varepsilon/2$ on each element of $G/\Stab ([m])$. By local finiteness of $\m F$, let now $A\subset \N$ finite, supporting a finite substructure $\m A$ of $\m F$, and such that $[m]\subset A$. Then $f$ is also constant up to $\varepsilon/2$ on each element of $G/\Stab (A)$. Because $f$ is also bounded, we can also find $\bar f : \funct{G}{\R^{p}}$ with finite range, constant on elements of $G/\Stab (A) $, and so that $\| f-\bar f\|_{\infty}<\varepsilon/2$. By Proposition \ref{prop:RP0}, there exists $g\in G$ such that $\bar f $ is constant on $gF$. Then $f$ is $\varepsilon$-constant on $gF$.     \qedhere 

\end{proof}

We can now show that $G$ is extremely amenable. Let $G \curvearrowright X$ be a continuous action, with $X$ compact. For $p\in \N$, $\phi : \funct{X}{\R^{p}}$ uniformly continuous and bounded, $F\subset G$ finite, $\varepsilon>0$, set $$ A_{\phi, \varepsilon, F} = \{ x\in X : \forall h \in F \quad \| \phi(h\cdot x) - \phi(x)\| \leq \varepsilon\}.$$ 

The family $(A_{\phi, \varepsilon, F})_{\phi, \varepsilon, F}$ is a family of closed subsets of $X$. We claim that it has the finite intersection property. Indeed, if $\phi_{1},\ldots, \phi_{l}, \varepsilon_{1},\ldots, \varepsilon_{l}, F_{1},\ldots, F_{l}$ are given, take $$f = (\phi_{1},\ldots,\phi_{l}), \quad \varepsilon = \min (\varepsilon_{1}, \ldots, \varepsilon_{l} ), \quad F = F_{1}^{-1}\cup\ldots\cup F_{l}^{-1}\cup \{e\}.$$ 

Fix $x\in X$ and consider the map $f : \funct{G}{\R^{p_{1}+\ldots+p_{l}}}$ defined by $$ \forall g\in G \quad f(g) = (\phi_{1}(g^{-1}\cdot x),\ldots,\phi_{l}(g^{-1}\cdot x))).$$

Because the maps $\phi_{i}$'s are uniformy continuous and the map $g\longmapsto g^{-1}\cdot x$ is left uniformly continuous (cf \cite{Pe}, p40), the map $f$ is left uniformly continuous. By Proposition \ref{prop:RP1}, there exists $g\in G$ so that $$ \forall h, h' \in F \quad \| f(gh) - f(gh')\|<\varepsilon.$$ 

Equivalently, $$\forall i\leq l \quad \forall h, h'\in F \quad \|\phi_{i}(h^{-1}g^{-1}\cdot x) - \phi_{i}(h'^{-1}g^{-1}\cdot x )\| \leq \varepsilon_{i} .$$

Taking $x_{0} = g^{-1}\cdot x$ and $h'=e$, we obtain $$\forall i\leq l \quad \forall h\in F_{i} \quad \|\phi_{i}(h\cdot x_{0}) - \phi_{i}(x_{0} )\| \leq \varepsilon_{i} .$$ 

This proves the finite intersection property of the family $(A_{\phi, \varepsilon, F})_{\phi, \varepsilon, F}$. By compactness of $X$, it follows that this family has a non empty intersection. Consider any element $x$ of this intersection. We claim that $x$ is fixed under the action of $G$: if not, we would find $g\in G$ so that $g\cdot x\neq x$. Then, there would be a uniformly continuous function $\phi_{0} : \funct{X}{[0,1]}$ so that $\phi_{0}(x)=0$ and $\phi_{0}(g\cdot x)= 1$. That would imply $x\notin A_{\phi_{0}, 1/2, \{g\}}$, a contradiction. \qedhere

\end{proof}

\section{Minimality and expansion property}

\label{section:min}

The purpose of this section is to prove Theorem \ref{thm:EP}.

\begin{defn}

\label{def:EP}

Let $\m F$ be a Fra\" iss\'e structure, and $\m F^{*}$ a precompact relational expansion of $\m F$. Say that $\Age(\m F^{*})$ has the \emph{expansion property} relative to $\Age(\m F)$ when for every  $ \m A \in \Age(\m F)$, there exists $ \m B \in \Age(\m F)$ such that $$\forall \m A^{*}, \m B^{*} \in \Age(\m F^{*}) \quad (\restrict{\m A^{*}}{L} = \m A \ \  \wedge \ \ \restrict{\m B^{*}}{L} = \m B) \Rightarrow \m A^{*} \leq \m B^{*}. $$
\end{defn}

When $\m A$ and $\m B$ are as above, we say that $\m B$ has the expansion property relative to $\m A$. Note that because $\Age (\m F)$ has the joint embedding property, the expansion property is equivalent to: $$\forall \m A^{*} \in \Age(\m F ^{*}) \quad \exists \m B \in \Age(\m F) \quad \forall \m B^{*} \in \Age(\m F^{*}) \quad (\restrict{\m B^{*}}{L} = \m B) \Rightarrow \m A^{*} \leq \m B^{*}. $$

For a fixed $\m A^{*} \in \Age(\m F^{*})$, any $\m B \in \Age(\m F)$ witnessing this property will also be said to have the expansion property relative to $\m A^{*}$. 

\

Here is a concrete example: consider the structure $\Q _2$ defined as $(\Q , Q_0 ,Q_1,<)$ where $\Q$ denotes the rationals, $<$ denotes the usual ordering on $\Q$, and $Q_0$, $Q_{1}$ are dense subsets of $\Q$. The appropriate language is $(P_{0}, P_{1}, <)$, made of two unary relations symbols and one binary relation symbol. The structure $\Q_{2}$ will play the role of $\m F^{*}$. For $\m F$, simply take the reduct $(\Q , Q_0 ,Q_1)$. Then $\Age(\m F^{*})$ is the class of all structures of the form $\m A =(A, P^{\m A}_{0}, P^{\m A} _{1}, <^{\m A})$ where $P^{\m A}_{0}, P^{\m A} _{1}$ partitions $A$ and $<^{\m A}$ is a linear ordering on $A$. This age does not have the expansion property relative to $\Age(\m F)$: if $\m A^{*} \in \Age(\m F^{*})$ is such that some element of $P^{\m A} _{1}$ is less than some element of $P^{\m A} _{0}$ and $\m B\in \Age(\m F)$, then there is an expansion $\m B^{*}$ of $\m B$ such that all elements of $P^{\m B} _{0}$ are less than those of $P^{\m B} _{1}$. In particular, $\m A^{*}$ does not embed in $\m B^{*}$. Thus, $\m B$ does not have the expansion property relative to $\m A^{*}$. However, if we consider only the class $\mathcal K$ of those elements $\m A^{*}$ of $\Age(\m F^{*})$ so that all elements of $P^{\m A} _{0}$ are less than those of $P^{\m A} _{1}$, then it is easy to see that the expansion property holds.

\setcounter{thmm}{3}

\begin{thmm}

Let $\m F$ be a Fra\" iss\'e structure, and $\m F^{*}$ precompact relational expansion of $\m F$. The following are equivalent: 

\begin{enumerate}

\item[i)] The flow $G\curvearrowright X^{*}$ is minimal. 

\item[ii)] $\Age(\m F^{*})$ has the expansion property relative to $\Age(\m F)$. 

\end{enumerate}

\end{thmm}

\noindent \textbf{Remark:} The structure $\m F^{*}$ does not have to be Fra\"iss\'e for this result to apply.

\

%\begin{prop}
%
%Let $\m F$ be a Fra\" iss\'e structure, and $\m F^{*}$ a precompact relational expansion of $\m F$. Assume that the flow $G\curvearrowright \overline{G\cdot\vec{R^{*}}}$ is minimal. Then $\Age(\m F^{*})$ has the expansion property relative to $\Age(\m F)$. 
%
%\end{prop}

From now on, we fix $\m F$ a Fra\" iss\'e structure and $\m F^{*}$ a precompact relational expansion of $\m F$.

\begin{prop}

\label{prop:age}

Let $\vec S, \vec T\in P^*$. Then $\vec S \in \overline{G\cdot \vec T}$ iff $\Age(\m F, \vec S) \subset \Age(\m F, \vec T)$. 

\end{prop}

\begin{proof}

Assume $\vec S \in \overline{G\cdot \vec T}$. Because $(\m F, g\cdot \vec T) \cong (\m F, \vec T)$ via $g$ for any $g\in G$, it suffices to show that for every finite set $A\subset \N$ supporting a substructure of $\m F$, there exists $g\in G$ so that $$\restrict{(\m F, g\cdot \vec T)}{A} = \restrict{(\m F, \vec S)}{A}.$$ 

But this is clearly implied by $\vec S\in \overline{G\cdot \vec T}$.

Conversely, consider a basic open neighborhood around $\vec S$ in $P^*$. By refining it, we may assume that it is given by a finite set $A\subset \N$ supporting a substructure of $\m F$. Because $\Age(\m F, \vec S) \subset \Age(\m F, \vec T)$, we can find $C \subset \N$ finite such that $$\restrict{(\m F, \vec T)}{C}\cong \restrict{(\m F, \vec S)}{A}.$$ 

Let $g:\funct{C}{A}$ witness this isomorphism. In particular, $g$ is an isomorphism between finite substructures of $\m F$. It can therefore be extended to some element $\hat g$ of $G$. Then $$\forall i \in I \quad \forall y_{1}\ldots y_{a(i)} \in A \quad S_{i}(y_{1}\ldots y_{a(i)}) \Leftrightarrow T_{i}(\hat g^{-1}(y_{1})\ldots \hat g^{-1}( y_{a(i)})).$$ 

Hence, $$\restrict{(\m F, \hat g\cdot \vec T)}{A} = \restrict{(\m F, \vec S)}{A}. \qquad \qedhere$$ 

\end{proof}

As it was the case for the Ramsey property, we now prove that the expansion property can be witnessed on $\m F$, as opposed to some finite substructure: 

\begin{prop}

\label{prop:EPcompact}

The class $\Age(\m F^{*})$ has the expansion property relative to $\Age(\m F)$ iff $\Age(\m F^{*}) \subset \Age(\m F, \vec S)$ for every $\vec S \in X^{*}$.
\end{prop}

\begin{proof}

Assume that $\Age(\m F^{*})$ has the expansion property relative to $\Age(\m F)$. Fix $ \m A^{*} \in \Age(\m F ^{*})$ and consider $\m B \in \Age(\m F)$ with the expansion property relative to $\m A^{*}$. Take now $\vec S \in X^{*}$. Then by Proposition \ref{prop:age}, $\restrict{(\m F, \vec S)}{B}\in \Age(\m F^{*})$, and because $\m B$ has the expansion property relative to $\m A^{*}$: $$\m A^{*} \leq \restrict{(\m F, \vec S)}{B}.$$

Therefore, $\m A^{*}\leq (\m F,\vec S)$ and $\Age(\m F^{*}) \subset \Age(\m F, \vec S)$. 

Conversely, assume that $\m F^{*}$ is Fra\"iss\'e and that $\Age(\m F^{*}) \subset \Age(\m F, \vec S)$ for every $\vec S \in X^{*}$. Let $\m A^{*} \in \Age(\m F^{*})$. For $C\subset \N$ finite and supporting a substructure of $\m F$, let $$X_{C} = \{ \vec S\in X^{*} : \m A^{*} \cong \restrict{(\m F, \vec S)}{C}\}.$$ 

Then, because $\Age(\m F^{*}) \subset \Age(\m F, \vec S)$ for every $\vec S \in X^{*}$, we have: $$X^{*} = \bigcup_{C\subset\N} X_{C}.$$ 

Compactness of $X^{*}$ allows then to find finite sets $C_{1},\ldots,C_{k}\subset \N$ such that $$ X^{*} = \bigcup_{j=1} ^{k} X_{C_{j}}.$$

Let $\m C$ denote the substructure of $\m F$ generated by $\bigcup_{j=1} ^{k} C_{j}$, and write $C$ the finite subset of $\N$ supporting it. We claim that $\m C$ has the expansion property relative to $\m A^{*}$: let $\m C^{*}$ be an expansion of $\m C$ in $\Age(\m F^{*})$. Consider an embedding $\phi : \funct{\m C^{*}}{\m F ^{*}}$. It induces an embedding between finite substructures of $\m F$ and can be extended to some element $g\in G$. Then, for every $i \in I$ and every $y_{1}\ldots y_{a(i)} \in C$, $$R^{\m C^{*}}_{i}(y_{1}\ldots y_{a(i)}) \Leftrightarrow R^{*}_{i}(g(y_{1})\ldots g( y_{a(i)}))\Leftrightarrow g^{-1}\cdot R^{*}_{i}(y_{1}\ldots y_{a(i)}).$$  

Therefore, setting $S_{i} = g^{-1}\cdot R^{*}_{i}$ for every $i\in I$, we obtain $\m C^{*} \cong \restrict{(\m F, \vec S)}{C}$. Now, $\vec S \in X_{C_{l}}$ for some $l\leq k$. So $$ \m A^{*} \cong \restrict{(\m F, \vec S)}{C_{l}} \subset \restrict{(\m F, \vec S)}{C}\cong \m C^{*}. \qquad \qedhere$$
\end{proof}

We now turn to a proof of Theorem \ref{thm:EP}. Minimality of $G\curvearrowright X^{*}$ is equivalent to $\vec R^{*} \in  \overline{G\cdot \vec S}$ for every $\vec S\in X^{*}$, which is in turn equivalent to (cf Proposition \ref{prop:age}) $\Age(\m F^{*}) \subset \Age(\m F, \vec S)$ for every $\vec S \in X^{*}$. Now, apply Proposition \ref{prop:EPcompact}.

\section{Universal minimal flows}

\label{section:UMF}

In this section, we prove Theorem \ref{thm:UMF}:

\begin{thmm}

\label{thm:UMF}

Let $\m F$ be a Fra\" iss\'e structure, and $\m F^{*}$ be a Fra\"iss\'e precompact relational expansion of $\m F$. Assume that $\Age(\m F^{*})$ consists of rigid elements. The following are equivalent: 

\begin{enumerate}

\item[i)] The flow $G\curvearrowright X^{*}$ is the universal minimal flow of $G$. 

\item[ii)] The class $\Age(\m F^{*})$ has the Ramsey property as well as the expansion property relative to $\Age(\m F)$. 

\end{enumerate}

\end{thmm}

\begin{proof}[Proof of $ii) \Rightarrow i)$]

Let $G\curvearrowright X$ be a compact minimal $G$-flow. It induces a continuous action $G^{*}\curvearrowright X$. By Theorem \ref{thm:EARP}, since $\Age(\m F^{*})$ has the Ramsey property and consists of rigid elements, $G^{*}$ is extremely amenable and $G^{*}\curvearrowright X$ has a $G^{*}$-fixed point, call it $\xi$. Let $p : \funct{G}{X}$ be defined by $$p(g)=g\cdot\xi$$ 

Then $p$ is right-uniformly continuous (cf \cite{Pe}, p.40) and is constant on elements of $G/G^{*}$: if $g^{-1}h\in G^{*}$, then $g^{-1}h\cdot \xi = \xi$ so $h\cdot\xi = g\cdot\xi$ ie $p(h)=p(g)$. So $p$ induces a right-uniformly continuous map $q : \funct{G/G^{*}}{X}$, which can be extended continuously to $\pi : \funct{X^{*}}{X}$ (we use here that $X^{*}$ is the completion of $G/G^{*}$ because $\m F^{*}$ is Fra\"iss\'e). This map is $G$-equivariant: this follows from continuity and from the fact that $q$ is $G$-equivariant. It is also onto because its range is a compact subset of $X$ containing the dense subset $G\cdot\xi$. \qedhere

\end{proof}

\begin{proof}[Proof of $i) \Rightarrow ii)$]

Because of minimality of $G\curvearrowright X^{*}$, Theorem \ref{thm:EP} ensures that $\Age(\m F^{*})$ has the expansion property relative to $\Age(\m F)$. To prove the Ramsey property of $\Age(\m F^{*})$, we are going to use the so-called \emph{Ramsey degrees}. Recall that the arrow relation $$ \m C \arrows{(\m B)}{\m A}{k, l}$$ means that for every map $c: \funct{\binom{\m C}{\m A}}{[k]}$, there is $\mc B \in \binom{\m C}{\m B}$ such that $c$ takes at most $l$-many values on $\binom{\mc B}{\m A}$.

\begin{defn}

For $\m A \in \Age(\m F)$, say that it has a \emph{finite Ramsey degree in $\Age(\m F)$} when there exists $l\in \N$ such that for every $k\in \N$ and $\m B \in \Age(\m F)$, there is $\m C\in  \Age(\m F)$ so that $$ \m C\arrows{(\m B)}{\m A}{k, l}.$$ 

The smallest possible value for $l$ is then the \emph{Ramsey degree of $\m A$ in $\Age(\m F)$}. 

\end{defn}

Here, we are going to show that every $\m A\in \Age(\m F)$ has a finite Ramsey degree, whose value is at most the number $t(\m A)$ of non-isomorphic expansions of $\m A$ in $\Age(\m F^{*})$. By compactness, it is enough to show that for every $k\in \N$ and $\m B \in \Age(\m F)$, $$ \m F\arrows{(\m B)}{\m A}{k, t(\m A)}.$$ 

Let $c:\funct{\binom{\m F}{\m A}}{[k]}$. Then $\overline{G\cdot c}$ is a compact subflow of $G\curvearrowright [k]^{\binom{\m F}{\m A}}$. As such, it admits a minimal subflow $X\subset \overline{G\cdot c}$. By i), there is a continuous $G$-equivariant map $\pi : X^{*} \twoheadrightarrow X$.

Recall that because $\m F^{*}$ is Fra\"iss\'e, $X^{*}$ is the completion of $\widehat{G/G^{*}}$ and so letting $\gamma := \pi ([e])$, we obtain that $\gamma$ is $G^{*}$-fixed. This implies that all copies of $\m A$ in $\m F$ that support isomorphic structures in $\m F^{*}$ have same $\gamma$-color, and that $\gamma$ takes no more than $t(\m A)$ values on $\binom{\m F}{\m A}$. Now, consider $\m B$. Because it is in $\Age(\m F)$, we may assume that it is supported by a finite subset $B$ of $\N$. Since $\gamma\in \overline{G\cdot c}$, we can find $g\in G$ such that $$\restrict{g\cdot c}{\binom{\m B}{\m A}} = \restrict{\gamma}{\binom{\m B}{\m A}}.$$

So $g\cdot c$ takes no more than $t(\m A)$ values on $\binom{\m B}{\m A}$, and so does $c$ on $\binom{g^{-1}(\m B)}{\m A}$. This proves that 
$\m A \in \Age(\m F)$ has a finite Ramsey degree in $\Age(\m F)$ at most equal to $t(\m A)$. To prove that $\Age(\m F^{*})$ has the Ramsey property, we will use the following general fact: 

\begin{prop}

\label{prop:Rdeg}

Let $\mathcal K^{*}$ be an expansion of $\Age(\m F)$ in $L^{*}$ satisfying the hereditarity property, the joint embedding property, the expansion property relative to $\Age(\m F)$, and such that every $\m A \in \Age(\m F)$ has a finite Ramsey degree in $\Age(\m F)$ whose value is at most the number of non-isomorphic expansions of $\m A$ in $\mathcal K^{*}$. Then $\mathcal K^{*}$ has the Ramsey property.  

\end{prop}

\begin{proof}[Proof of Proposition \ref{prop:Rdeg}]

Fix $k\in \N$, $\m A^{*}, \m B^{*} \in \mathcal K^{*}$. Consider $\m D$ with the expansion property with respect to $\restrict{\m A^{*}}{L}$ and to $\m B^{*}$ (finding such a structure is easy thanks to the expansion property of $\mathcal K^{*}$ relative to $\Age(\m F)$ and to the joint embedding property of $\mathcal K^{*}$). Consider also $\m C \in \Age(\m F)$ such that $$ \m C\arrows{(\m D)}{\m A}{kt(\m A), t(\m A)}.$$ 

Let $\m C^{*}\in \mathcal K^{*}$ be any expansion of $\m C$, and fix $c^{*}:\funct{\binom{\m C^{*}}{\m A^{*}}}{[k]}$. Seeing $\binom{\m C^{*}}{\m A^{*}}$ as a subset of $\binom{\m C}{\m A}$, extend $c^{*}$ to a coloring $c : \funct{\binom{\m C}{\m A}}{[k]}$, and define a new coloring $\bar c : \funct{\binom{\m C}{\m A}}{[k]\times [t(\m A)]}$ by $\bar c(\mc A) = (c(\mc A), i(\mc A))$, where $i(\mc A)$ denotes the isomorphism type of $\mc A$ seen as substructure of $\m C^{*}$. Because $\m A$ has finite Ramsey degree at most equal to $t(\m A)$ in $\Age(\m F)$, there is a copy $\mc D$ of $\m D$ in $\m C$ with at most $t(\m A)$ many $\bar c$-colors appearing on $\binom{\mc D}{\m A}$. Because $\m D$ has the expansion property relative to $\m A$, all isomorphism types of $\m A$ in $\Age(\m F^{*})$ appear in $\binom{\mc D}{\m A}$. Therefore, exactly $t(\m A)$ many $\bar c$-colors appear on $\binom{\mc D}{\m A}$, and all copies of $\m A$ sharing the same isomorphism type receive the same $\bar c$-color, hence the same $c$-color. In particular, all copies of $\m A^{*}$ in $\binom{\mc D}{\m A}$ receive the same $c^{*}$-color. Now, using the expansion property of $\m D$ relative to $\m B^{*}$, $\mc D$ contains a copy $\mc B ^{*}$ of $\m B^{*}$, and $\binom{\mc B^{*}}{\m A^{*}}$ is $c^{*}$-monochromatic.  \qedhere   

\end{proof}

Applying the previous proposition to the class $\mathcal K^{*} = \Age(\m F^{*})$ shows that it has the Ramsey property. This completes the proof of $ii)\Rightarrow i)$. \qedhere

\end{proof}

\section{Subclasses of ages with the Ramsey property}

\label{section:subclass}

The goal of this section is to prove Theorem \ref{thm:subclass}:

\begin{thmm}

\label{thm:subclass}

Let $\m F$ be a Fra\" iss\'e structure, and $\m F^{*}$ be a Fra\"iss\'e precompact relational expansion of $\m F$. Assume that $\Age(\m F^{*})$ consists of rigid elements, and has the Ramsey property. Then $\Age(\m F^{*})$ admits a Fra\"iss\'e subclass with the Ramsey property and the expansion property relative to $\Age(\m F)$.  

\end{thmm}

\begin{proof}

Let $\vec S \in \overline{G\cdot \vec R^{*}}$ be such that $G\curvearrowright\overline{G\cdot \vec S}$ is minimal. Note that thanks to Proposition \ref{prop:age}, $\Age(\m F, \vec S)\subset \Age(\m F^{*})$ because $\vec S \in  \overline{G\cdot \vec R^{*}}$. We claim that $\Age(\m F, \vec S)$ is a required. 

The expansion property comes from minimality of $G\curvearrowright\overline{G\cdot \vec S}$ (Theorem \ref{thm:EP}). 

As for the Ramsey property, the proof follows the same pattern as the proof of Theorem \ref{thm:UMF}, $i)\Rightarrow ii)$: first, we use the Ramsey property for $\Age(\m F^{*})$ and the extreme amenability of $G^{*}$ to show that for every $\m A, \m B \in \Age(\m F)$, and every $k$-coloring $c$ of $\binom{\m F}{\m A}$, there is $g\in G$ so that on $\binom{\m B}{\m A}$, the value of $g\cdot c (\mc A)$ depends only on the isomorphism type of $\mc A$ seen as a substructure of $\m F^{*}$. Observe however that choosing $\m B$ so that the substructure it supports in $\m F^{*}$ is equal to the substructure it supports in $(\m F, \vec S)$ (this is possible because $\Age(\m F, \vec S)\subset \Age(\m F^{*})$), and using the fact that $\vec S$ is $G^{*}$ fixed, we can actually make sure that the value of $g\cdot c (\mc A)$ depends only on the isomorphism type of $\mc A$ seen as a substructure of $(\m F, \vec S)$ (and not only of $\m F^{*}$). It follows that every $\m A \in \Age(\m F)$ has a finite Ramsey degree less or equal to the number of non-isomorphic expansions of $\m A$ in $\Age(\m F, \vec S)$, and not only in $\Age(\m F^{*})$. The Ramsey property for $\Age(\m F, \vec S)$ is then derived from Proposition \ref{prop:Rdeg}. 

Finally, for classes consisting of rigid elements, the hereditarity, joint embedding and Ramsey properties imply the amalgamation property. This crucial fact was first noticed by Ne\v set\v ril, but we include it here for completeness (we repeat the argument given in \cite{KPT}, first half of p.129). Fix $\m A, \m B, \m C \in \Age(\m F, \vec S)$ and embeddings $f: \funct{\m A}{\m C}$, $g:\funct{\m A}{\m C}$. By the joint embedding property, find $\m E \in \Age(\m F, \vec S)$ in which both $\m B$ and $\m C$ can be embedded. Then, thanks to the Ramsey property, find $\m D \in \Age(\m F, \vec S)$ such that $$ \m D \arrows{(\m E)}{\m A}{4},$$ and consider the coloring $c : \funct{\binom{\m D}{\m A}}{\{ x : x\subset \{ \m B, \m C\}\}}$ defined as follows: given $\mc A\in \binom{\m D}{\m A}$, declare that $\m B\in c(\mc A)$ iff there is an embedding $r:\funct{\m B}{\m D}$ so that $r\circ f (\m A) = \mc A$, and similarly for $\m C$. Let $\mc E$ be such that $\binom{\mc E}{\m A}$ is $c$-monochromatic. Because both $\m B$ and $\m C$ embed in $\m E$, the corresponding constant value of $c$ is $\{ \m B, \m C\}$. Now, for $\mc A \in \binom{\mc E}{\m A}$, there are $r:\funct{\m B}{\m D}$ and and $s : \funct{\m C}{\m D}$ so that $r\circ f (\m A) = \mc A$ and $s\circ g (\m A) = \mc A$. So $r\circ f$ and $s\circ g$ are isomorphisms of $\m A$ with $\mc A$. Since $\m A$ and $\mc A$ are rigid, it follows that $r\circ f = s\circ g$, and $\m D, r$ and $s$ witness the amalgamation property. It follows that $\Age(F, \vec S)$ is Fra\"iss\'e. \qedhere 

\end{proof}

Here is an example to illustrate Theorem \ref{thm:subclass}. Consider the structures already described in Section \ref{section:min}: $\m F^{*} = \Q_{2} = (\Q, Q_{0}, Q_{1}, <)$, $\m F = (\Q, Q_{0}, Q_{1})$. Then $\Age(\m F^{*})$ has the Ramsey property (this fact is, for example, proved in \cite{KPT}), but we have seen that it does not have the expansion property relative to $\Age(\m F)$. However, passing to $\mathcal K \subset \Age(\m F^{*})$ consisting of all those $\m A^{*} \in \Age(\m F^{*})$ so that all elements of $P^{\m A} _{0}$ are less than those of $P^{\m A} _{1}$, then the expansion property holds. As for the Ramsey property, it holds thanks to the classical product Ramsey theorem. 

\section{The universal minimal flow of the circular directed graph $\m S(2)$}

\label{section:Cy}

The tournament $\m S(2)$, called the dense local order, is defined as follows: let $\mathbb{T}$ denote the unit circle in the complex plane. Define an oriented graph structure on $\mathbb{T}$ by declaring that there is an arc from $x$ to $y$ (in symbols, $y\la{\T}x$) iff $0 < \arg (y/x) < \pi$. Call $\Ty$ the resulting oriented graph. The dense local order is then the substructure $\m S(2)$ of $\Ty$ whose vertices are those points of $\mathbb T$ with rational argument. It is represented in the picture below.

\begin{figure}[h]
\begin{picture}(0,40)(0,-20)

\put(0,0){\circle{40}}
\put(0,20){\circle*{3}}
\put(10,-17){\circle*{3}}
\put(-20,0){\circle*{3}}

\put(0,20){\vector(-1,-1){20}}
\put(10,-17){\vector(-1,4){9}}
\put(-20,0){\vector(2,-1){30}}

\end{picture}
\caption{The tournament $\Cy$}
\end{figure}

This structure is one of the only three countable ultrahomogeneous tournaments (a tournament is a directed graph where every pair of distinct points supports exactly one arc), the two other ones being the rationals $(\Q, <)$, seen as a directed graph where $x \la{\Q} y$ iff $x<y$, and the so-called random tournament. It is therefore a Fra\"iss\'e structure in the language $L = \{ \leftarrow\}$ consisting of one binary relation. More information about this object can be found in \cite{W}, \cite{La} or \cite{Ch}. 

Our goal in this section is to describe the universal minimal flow of its automorphism group. But before doing so, let us mention the following simple fact, which shows why the technique developed in \cite{KPT} cannot be applied in order to do so:   

\begin{prop}
No pure order expansion of $\Cy$ has an age with the Ramsey and the expansion property. 
\end{prop}

\begin{proof}
A simple way to achieve this is to use Ramsey degrees. A general fact is indeed that for every expansion of $\Cy$ whose age $\mathcal K^{*}$ has the Ramsey and the expansion property, the Ramsey degree of any $\m A$ in $\Age(\Cy)$ is exactly equal to the number of non-isomorphic expansions of $\m A$ in $\mathcal K^{*}$. If $\Cy ^{*}$ were a pure order expansion of $\Cy$, the substructure of $\Cy$ consisting of a single point would have only one expansion, and so its Ramsey degree in $\Age(\Cy)$ would be equal to $1$. This, however, is false, because coloring $\Cy$ with left and right half provides a coloring with no monochromatic cyclic triangle.  
\end{proof}

However, in view of Theorem \ref{thm:UMF}, it suffices to find a precompact relational expansion $\m S(2)^{*}$ of $\m S(2)$ whose age has the Ramsey property and the expansion property. It turns out that such an expansion essentially appears in \cite{LNS}, where the finite and the infinite Ramsey properties of $\m S(2)$ were analyzed. We now turn to a description of $\Cy ^{*}$. The appropriate language is $$L^{*} = L\cup\{P_{j}:j \in [2]\},$$ every symbol $P_{j}$ being unary. We expand $\Cy$ as $(\m S(2), P_{0} ^{*}, P_{1}^{*})$ in $L^{*}$, where $P^{*} _{0} (x)$ holds iff $x$ is in the right half plane, and $P^{*}_{1}(x)$ iff it is in the left half plane. Quite clearly, $\Cy ^{*}$ is a precompact relational expansion of $\Cy$.

\begin{prop}
The class $\Age(\Cy ^{*})$ has the Ramsey property and the expansion property relative to $\Age(\Cy)$. 
\end{prop}

\begin{proof}

The crucial fact here is the link that $\Cy ^{*}$ possesses with the structure $\Q_{2}$, which we have already encountered. In general, when $n\in \N$, the structure $\Q _n$ is defined as $(\Q , Q_0 ,\ldots,Q_{n-1},<)$ where $\Q$ denotes the rationals, $<$ denotes the usual ordering on $\Q$, and every $Q_i$ is a dense subset of $\Q$. As before for $\Q$, it is convenient to think as $\Q_{n}$ as a directed graph together with some partition, so $\Q_{2}$ can really be seen as a structure in the language $L^{*}$. The link between $\Cy^{*}$ and $\Q_{2}$ is the following one: the structure $\Q_{2}$ is simply obtained from $\Cy^{*}$ by reversing all the arcs whose extremities do not belong to the same part of the partition. The simple reason behind that fact is that if $x, y\in \Cy$ are such that $P^{*}_{0}(x)$ and $P^{*}_{1}(y)$, then $x\la{\T}y$ iff $(-y)\la{\T}x$, where $(-y)$ denotes the opposite of $y$. So one way to realize the transformation from $\Cy^{*}$ to $\Q_{2}$ is to consider $\Cy^{*}$, to keep the partition relation, but to replace the arc relation by the relation obtained by symmetrizing all the elements in the left half. Quite clearly, the new arc relation defines a total order, which is dense in itself and without extremity point, and where both parts of the partition are dense. Therefore, the resulting structure is $\Q_{2}$. Similarly, applying the same transformation to $\Q_{2}$ gives raise to $\Cy^{*}$. Formally, $\Cy ^{*}$ and $\Q_{2}$ are said to be \emph{first-order simply bi-definable}. So, in some sense, $\Q_{2}$ and $\Cy^{*}$ really can be though as the same structure. Using this, it is easy to see that the Ramsey property holds for the $\Age (\Cy^{*})$ iff it holds for $\Age(\Q_{2})$, a fact which is proved in \cite{KPT} (in fact, it is shown that  $\Age(\Q_{n})$ has the Ramsey property for every $n\geq 1$). 

As for the expansion property, it is essentially proved in \cite{LNS}, Lemma 3, where the result is actually proved for structures called \emph{extensions} of elements of $\Age(\Cy)$. By definition, an extension of $\m A \in \Age(\Cy)$ is a substructure of $\Q_{2}$ obtained from $\m A$ by first taking an  expansion of $\m A$ in $\Age(\Cy^{*})$, and then turning it into a substructure of $\Q_{2}$ by applying the transformation we described above between $\Cy^{*}$ and $\Q_{2}$. Lemma 3 of \cite{LNS} then asserts that for every $\m A \in \Age(\Cy)$, there exists $\m B \in \Age(\Cy)$ so that every extension of $\m A$ embeds in every extension of $\m B$. This result is equivalent to the expansion property of $\Age(\Cy ^{*})$ relative to $\Age(\Cy)$. \qedhere

\end{proof}

Setting $G=\Aut(\Cy)$ and $G^{*}=\Aut(\Cy^{*})$, the universal minimal flow of $G$ is, in virtue of Theorem \ref{thm:UMF}, the action $G\curvearrowright X^{*}$, where $X^{*}:=\overline{G\cdot(P^{*}_{0}, P^{*}_{1})}$, the closure of $G\cdot(P^{*}_{0}, P^{*}_{1})$ in $[2]^{\Cy}\times[2]^{\Cy}$. We are going to provide a concrete description of that action. In the unit circle $\T$, consider the set $S$ supporting $\Cy$, the set $(-S)$ of all its opposite points, and the set $C = \T \smallsetminus (S\cup (-S))$. Consider $$\hat \T = C\cup \left((S\cup (-S))\times [2]\right).$$ 

Intuitively, it is obtained from the unit circle $\T$ by doubling the points in $S\cup (-S)$. Next, for $t\in \hat \T$, define $p(t)$ as the natural projection of $t$ on $\T$, and for $\alpha, \beta \in S\cup(-S)$ so that $\alpha \la{\T}\beta$, define $[\alpha, \beta]$ by: $$ [\alpha, \beta]:= \{ (\alpha, 0)\} \cup \{ t\in \hat \T: \alpha \la{\T}p(t)\la{\T}\beta \}\cup \{(\beta, 1)\}.$$

This set is represented in Figure 2 (as the right part of the circle, together with the two black dots).

\begin{figure}[h]
\begin{picture}(0,60)(0,-30)

\put(0,0){\circle{40}}

%\put(-20,-20){\line(1,1){40}}

%\put(-30,-15){$T_1$}
%\put(-17,-28){$T_2$}

%\put(14,14){\circle*{3}}
%\put(10,4){$t$}

\put(0,20){\circle*{3}}
\put(-2,10){$\alpha$}

\put(14,-14){\circle*{3}}
\put(10,-10){$\beta$}

%\put(20,0){\circle*{3}}
%\put(22,-2){$u$}

%\put(-150,20){$\bullet$ $a, b$ rational}

%\put(-150,0){$\bullet$ $u=(b,1)$ allowed}

\put(0,20){\circle*{3}}

\put(14,28){$0$}
\put(11,31){\circle*{3}}
\put(0,20){\line(1,1){10}}

%\put(-140,-10){$u=(b,2)$ not allowed}

\put(-18.5,29.5){$1$}
\put(-11,31){\circle{3}}
\put(0,20){\line(-1,1){10}}

%\put(-150,-30){$\bullet$ $u=(a,2)$ allowed}

\put(14,-14){\circle*{3}}
\put(24,-14){\circle*{3}}
\put(26,-16){$1$}
\put(14,-14){\line(1,0){10}}

%\put(-140,-40){$u=(a,1)$ not allowed}

\put(14,-24){\circle{3}}
\put(12,-34){0}
\put(14,-14){\line(0,-1){8.5}}

\end{picture}
\caption{The set $[\alpha, \beta]$}
\end{figure}
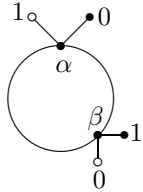

\begin{prop}

Sets of the form $[\alpha, \beta]$ form a basis of open sets for a topology on $\hat \T$, and the corresponding space is homeomorphic to $X^{*}$.

\end{prop}

\begin{proof}
The group $G$ acts naturally on $(P^{*}_{0}, P^{*}_{1})$ by moving $(P^{*}_{0}, P^{*}_{1})$. It follows that any element $(T_{0}, T_{1})$ in the orbit of $(P^{*}_{0}, P^{*}_{1})$ partitions $\Cy$ into two halves whose extremity points are not $\Cy$. Such an element is coded by $t := \sup T_0 \in C$ (this supremum is justified by the fact that the directed graph relation totally orders $T_{0}$), see Figure 3.

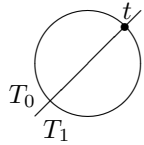
\begin{figure}[h]
\begin{picture}(0,50)(0,-30)

\put(0,0){\circle{40}}

%\put(0,-30){\line(0,1){60}}

\put(-20,-20){\line(1,1){40}}

\put(-30,-15){$T_0$}
\put(-17,-28){$T_1$}

\put(14,14){\circle*{3}}
\put(13,17){$t$}

\end{picture}
\caption{The point $t$ associated to $(T_{0}, T_{1})$.}
\end{figure}

The topology induced by $[2]^{\Cy}\times[2]^{\Cy}$ can be described as follows: a basic open set around $(T_0,T_1)$ is provided by a finite set $F\subset \Cy$ and defined by: 
$$(U_0, U_1) \in U_{F} \ \textrm{iff} \ (T_0, T_1) \ \textrm{and} \ (U_0, U_1)  \ \textrm{partition $F$ the same way.}$$ 

We now turn into a description of the closure $X^{*}$ of the orbit of $(P^{*}_{0}, P^{*}_{1})$. According to Proposition \ref{prop:age}, it consists of all those elements $(T_{0}, T_{1})\in [2]^{\Cy}\times[2]^{\Cy}$ so that $$\Age(\Cy, T_{0}, T_{1}) \subset \Age(\Cy^{*}).$$ 

Therefore, those are partitions of $\Cy$, and each of the parts has to be totally ordered by the arc relation on $\Cy$. The element $t$ defined as before may take any value in $\T$ because of the denseness of $\Cy$. However, $t$ or its opposite $(-t)$ can belong to $\Cy$. Equivalently, we may have $t\in S\cup (-S)$.  When that happens, $t$ does not suffice to characterize $(T_{0}, T_{1})$, as there are two choices for $(T_0,T_1)$: if $t\in S$, then either $t\in T_0$, and in that case we code $(T_0,T_1)$ by $(t,0)$; or $t\in T_1$, and we code $(T_0,T_1)$ by $(t,1)$. If $t\in (-S)$, then either $-t\in T_0$, and in that case we code $(T_0,T_1)$ by $(t,1)$; or $-t\in T_1$, and we code $(T_0,T_1)$ by $(t,0)$. It follows that as a set, we may think of $X^{*}$ as $\hat \T$, and that the group $G$ acts naturally on $X^{*}$.

To finish the proof, it suffices to show that the sets of the form $[\alpha, \beta]$ actually correspond to the basic open sets in $X^{*}$. We consider the case $\alpha \in S$ and $\beta \in (-S)$. The other cases are treated similarly. Let $a=\alpha$ and $b=-\beta$. Those two points are in $\Cy$. Consider now $F=\{ a, b\}$ and take $t \in \hat \T$ so that $$\alpha \la{\T}p(t)\la{\T}\beta.$$

It is then easy to see that $[\alpha, \beta]$ is equal to the basic open set $O$ around $t$ based on $F$. In fact, it turns out that all basic open sets are of that form. Indeed, consider now $F\subset \Cy$ finite and $t\in \hat \T$. The set $F\cup (-F)$ subdivides $\T$ into intervals, and this subdivision gives raise to a unique partition of $\hat \T$ into intervals. Let $\alpha, \beta$ denote the only elements of $F\cup (-F)$ so that $t\in [\alpha, \beta]$. Then observe that the basic open set $O$ around $t$ and based on $F$ is actually equal to $[\alpha, \beta]$. \qedhere

\end{proof}

This representation has at least two advantages. First, it is clear that $X^{*}$ is homeomorphic to the Cantor space. Second, it allows to visualize pretty well the action of $G$ on $X^{*}$, which is not so common when dealing with universal minimal flows. Another remarkable instance where that happens is due to Pestov in \cite{Pe1}. It deals with the orientation preserving homeomorphisms of $\T$, equipped with the pointwise convergence topology. That example provided the first known example of a metrizable, non-trivial, universal minimal flow, which is, in that case, the natural action on the circle by homeomorphisms.   

\section{The universal minimal flow of the circular directed graph $\m S(3)$}

\label{section:Cz}

The technique used in the previous section in order to compute the universal minimal flow of the dense local order also applies in the case of another directed graph, called $\Cz$. The notation suggests that $\Cz$ is a modified version of $\Cy$, and it is indeed the case. Call $\Dz = (\T, \la{\mathbb D})$ the directed graph defined on $\T$ by declaring that there is an arc from $x$ to $y$ iff $0 < \arg (y/x) < 2\pi/3$. The directed graph $\Cz$ is then the substructure of $\Dz$ whose vertices are those points of $\mathbb T$ with rational argument. It is represented in the picture below.

\begin{figure}[h]
\begin{picture}(0,40)(0,-20)

\put(0,0){\circle{40}}
\put(0,20){\circle*{3}}
\put(10,-17){\circle*{3}}
\put(-20,0){\circle*{3}}

\put(0,20){\vector(-1,-1){20}}
\put(-20,0){\vector(2,-1){30}}

\end{picture}
\caption{The directed graph $\Cz$}
\end{figure}

Like $\Cy$, $\Cz$ is Fra\"iss\'e structure in the language $L = \{ \leftarrow\}$ consisting of one binary relation. The main obvious difference with $\Cy$ is that it is not a tournament (that is, some pairs of points do not support any arc). For more information about this object, we refer to \cite{Ch}. 

For the same reason as in the case of $\Cy$, no pure order expansion of $\Cz$ has an age with both the Ramsey and the expansion property, but there is a precompact expansion $\Cz^{*}$ which does. The appropriate language is $$L^{*} = L\cup\{P_{j}:j \in [3]\},$$ with every symbol $P_{i}$ unary, and $\Cz^{*}$ is defined by $\Cz^{*} = (\Cz, P^{*}_{0}, P^{*}_{1}, P^{*}_{2})$, where $$P^{*}_{j}(x) \Leftrightarrow \left( \frac{2j\pi}{3} < \arg(x) + \frac{\pi}{6} < \frac{2(j+1)\pi}{3}\right)$$

The corresponding structure is described below:

\begin{figure}[h]
\begin{picture}(0,45)(0,-20)

\put(0,0){\circle{40}}

\put(0,0){\line(0,1){30}}
\put(0,0){\line(6,-5){22}}
\put(0,0){\line(-6,-5){22}}

\put(-13,23){$P^{*}_1$}
\put(3,23){$P^{*}_0$}
\put(-3,-30){$P^{*}_2$}

\end{picture}
\caption{The expansion $\Cz^{*}$}
\end{figure}

\begin{prop}
The class $\Age(\Cz ^{*})$ has the Ramsey property. 
\end{prop}

\begin{proof}

Like in the case of $\Cy ^{*}$, $\Cz^{*}$ has a simply bi-definable with a structure of the family $(\Q_{n})_{n\in\N}$, namely $\Q_{3}$. To see this, it will be convenient to think that the relations $P^{*}_{0}, P^{*}_{1}, P^{*}_{2}$ are actually indexed by $\Z/3\Z$. The transformation that allows to go from $\Cz^{*}$ to $\Q_{3}$ is the following one. Keep the same partition relations. Let $x, y \in \Cz$. If they belong to the same part, do not change the arc relation. If $P^{*}_{j}(x)$ and $P^{*}_{j+1}(y)$, where $j \in \Z/3\Z$, then either $y\la{\Cz}x$ or there is no arc between $x$ and $y$. In the first case, reverse the arc. In the second one, create an arc from $x$ to $y$. The resulting structure is then $\Q_{3}$. Again, there is a simple geometric reason to explain why that works: call $r$ the rotation about the origin with angle $(-2\pi/3)$. Then, if $x, y\in \Cz$ are such that $P^{*}_{j}(x)$ and $P^{*}_{j+1}(y)$, then, in $\Dz$, $y\la{\D}x$ iff $x\la{\D}r(y)$. Therefore, one way to realize the transformation described above is to preserve the partition relation, but to replace the arc relation by the new arc relation obtained by applying $r$ to all the elements of $P^{*}_{1}$, and $r^{-1}$ to all the elements of $P^{*}_{2}$. This new arc relation provides a total order which is dense in itself and without extremity points. Furthermore, all parts of the partition are dense. It follows that the resulting structure is $\Q_{3}$. Because $\Age(\Q_{3})$ has the Ramsey property, so does $\Age(\Cz ^{*})$.  \qedhere

\end{proof}

Consider now $G = \Aut(\Cz)$, as well as the flow $G\curvearrowright X^{*}$ where $$X^{*}=\overline{G\cdot(P^{*}_{0}, P^{*}_{1}, P^{*}_{2})}\subset  ([2]^{\Cy})^{3}.$$

\begin{prop}
The flow $G\curvearrowright X^{*}$ is minimal. 
\end{prop}

\begin{proof}

Using the same analysis as in the previous section, it is easy to see that the elements of $X^{*}$ are the partitions of $\Cz$ into three parts, each of them being totally ordered by the arc relation and therefore of angular diameter $2\pi/3$, with one of the extremity points possibly in $\Cz$. Therefore, all of them have an age equal to $\Age(\Cz ^{*})$. By Proposition \ref{prop:EPcompact}, it follows that the class $\Age(\Cz ^{*})$ has the expansion property relative to $\Age(\Cz)$, and by Theorem  \ref{thm:EP}, the flow $G\curvearrowright X^{*}$ is minimal.  \qedhere

\end{proof}

As a result, every element in $\Age(\Cz)$ has a finite Ramsey degree in $\Age(\Cz)$, equal to the number of its non-isomorphic expansions in $\Age(\Cz ^{*})$. In fact, it turns out that computing the exact value of this number is possible thanks to the same method as the one used in \cite{LNS} to compute Ramsey degrees in $\Age(\Cy)$. For a given $\m A \in \Age(\Cz)$, it is equal to $$3|\m{A}|/|\Aut(\m{A})|.$$

Another consequence is of course that $G\curvearrowright X^{*}$ is the universal minimal flow of $G$. As previously, a concrete realization of this action is available. The same kind of argument as in the previous section leads to the following description: recall that $S$ denotes the underlying set of $\Cy$ and $\Cz$, and that $r$ denotes the rotation about the origin with angle $(-2\pi/3)$. Let $E=\T\smallsetminus(S\cup r(S)\cup r^{-1}(S))$, and set $$\tilde{\T} = E\cup((S\cup r(S)\cup r^{-1}(S))\times[2])$$ 

As before, the right way to think about $\tilde \T$ is as $\T$, with certain points doubled. Let $q$ be the natural projection from $E$ onto $\T$, and define
$$ [\alpha, \beta]:= \{ (\alpha, 0)\} \cup \{ t\in \tilde \T: \alpha \la{\D}p(t)\la{\D}\beta \}\cup \{(\beta, 1)\}.$$

The set $X^{*}$ can be identified with $\tilde \T$ as follows. For $\vec T = (T_0,T_1, T_{2})\in X^{*}$, let $t = \sup T_0$. Define a map $\phi (\vec T)$ by:

\begin{enumerate}

\item If $t\in E$, $\phi (\vec T) = t$. 

\item If $t\in S$, $\phi (\vec T)$ equals $(t,0)$ if $t\in T_{0}$, and $(t,1)$ if $t\in T_{2}$. 

\item If $t\in r(S)$, $\phi (\vec T)$ equals $(t, 0)$ if $r^{-1}(t)\in T_{1}$, and $(t,1)$ if $r^{-1}(t)\in T_{0}$. 

\item If $t\in r^{-1}(S)$, $\phi (\vec T)$ equals $(t, 0)$ if $r(t)\in T_{2}$, and $(t,1)$ if $r(t)\in T_{1}$. 

\end{enumerate}

The map $\phi$ is a bijection, and it consequently allows $G$ to act on naturally on $\tilde \T$.  Then, as previously, one can prove that the sets of the form $[\alpha, \beta]$ are the images of the basic open sets in $X^{*}$. Thus, for the corresponding topology, the action $G\curvearrowright \tilde \T$ is continuous, and the map $\phi:\funct{X^{*}}{\tilde \T}$ is an isomorphism of $G$-flows. In other words:

\begin{prop}
The action $G\curvearrowright \tilde \T$ is the universal minimal flow of $\Aut(\Cz)$.

\end{prop}

\section{The relevance of precompact relational expansions}

\label{section:why}

The purpose of this section is to discuss the status of precompact relational expansions as the relevant framework for generalizing \cite{KPT}. Let us first explain why we only need to deal with \emph{relational} expansions (where $L^{*}\smallsetminus L$ consists only of relation symbols), as opposed to more general ones, where function symbols may be incorporated. The reason is that, given a Fra\"iss\'e structure $\m F$, if one is able to find a Fra\"iss\'e expansion $\m F^{*}$ such that $\Age(\m F)$ has the Ramsey property and consists of rigid elements, and such that the quotient $\Aut(\m F)/\Aut(\m F^{*})$ is precompact, then there is a Fra\"iss\'e relational expansion $\m F^{*, rel}$ with the same properties. Indeed, the group $G^{*}=\Aut(\m F^{*})$ is a closed subgroup of $\Aut(\m F)$, so there is a relational Fra\"iss\'e expansion $\m F^{*, rel}$ of $\m F$ such that $\Aut(\m F^{*,rel})=G^{*}$. Then, $\m F^{*, rel}$ is as required. Precompactness for $\Aut(\m F)/\Aut(\m F^{*, rel})$ is obvious. Next, the group $G^{*}$ is extremely amenable thanks to Theorem  \ref{thm:EARP}, and it follows, still thanks to Theorem \ref{thm:EARP}, that $\Age(\m F^{*,rel})$ consists of rigid elements and has the Ramsey property. Working with relational expansions is therefore enough. This is quite fortunate because in the general case, many notions from \cite{KPT} do not transfer as easily as they do in the relational setting. Next, as for the role of the precompactness assumption, it is simply because both Theorems \ref{thm:EP} and \ref{thm:UMF} fail without it. On the other hand, in view of the smooth transfer from Theorems \ref{thm:OP} and \ref{thm:KPTUMF} to Theorems \ref{thm:EP} and \ref{thm:UMF}, and of the link it allows with the right uniform structure on the automorphism group $G$ and its quotient $G/G^{*}$, precompactness appears more naturally than, say, finite relational expansions. From the practical point of view, this intuition materializes with the following example: consider the language made of a binary relation symbol $<$ together with countably many relational symbols $E_{n}$, $n\geq 1$, with arity $2n$. Consider $\mathcal K$ the class of all finite structures in that language, where $<$ is interpreted as a linear order, and each $E_{n}$ is interpreted as an equivalence relation with at most two classes on the set of $n$-uples. Then $\mathcal K$ is the age of a Fra\"iss\'e structure $\m F$, and its universal minimal flow can be computed thanks to the expansion $\m F^{*}$ obtained from $\m F$ by adding a unary relation symbol for each equivalence class in $\m F$. In other words, precompact relational expansions, as opposed to finite ones, are sometimes really necessary. However, it could still be that finite relational expansions do suffice when the structure $\m F$ we start with has a finite language. Future research will probably help deciding whether this is a general phenomenon or not. Note that, as mentioned previously, our technique also allows to compute the universal minimal flow for all the groups coming from countable ultrahomogeneous graphs, tournaments and posets. For graphs, all the results are essentially obtained in \cite{KPT}. The case of the disjoint union of countably many disjoint copies of a fixed complete graph $K_{n}$ is not explicitly treated there, but, as observed by Soki\'c in \cite{So2} in a slightly different context, it can be obtained thanks to a simple description of the automorphism group. For tournaments, \cite{KPT} covers the case of $(\Q,<)$ and of the universal tournament, and we presented in Section \ref{section:Cy} the case of the dense local order. For posets, all the results are included in \cite{So1} and \cite{So2}. Furthermore, some recent work in collaboration with Jakub Jasi\'nski, Claude Laflamme and Robert Woodrow suggests that finite relational expansions are indeed sufficient in the case of all countable ultrahomogeneous directed graphs, whose classification was made available by Cherlin in \cite{Ch}. Finally, note that in fact, for a closed subgroup $G$ of $S_{\infty}$, $M(G)$ can always be computed thanks to Theorem \ref{thm:UMF}, provided $M(G)$ is metrizable and has a $G_{\delta}$ orbit. This result will appear in a subsequent paper in collaboration with Todor Tsankov.

%Let us now finish with a few chronological details. Even though \cite{KPT} already contained examples where pure order expansions do not directly lead to the universal minimal flow, a little bit of extra work always allowed to reach it easily. For that reason, the need to generalize the techniques did not appear there, but only in 2008, when partition properties of $\Cy$ were analyzed (cf \cite{LNS}), and when it became clear that a slight modification of \cite{KPT} would suffice to compute the corresponding universal minimal flow. As for all the previous cases, the modification was carried out ``by hand'', and the results announced in Urbana-Champaign, April 2008. However, it was already felt at that time that a true generalization of \cite{KPT}, as opposed to a case by case analysis, would be needed at some point (an intuition which turned out to be confirmed by the works \cite{So2} of Soki\'c, and \cite{J} of Jasi\'nski). The case of finite relational expansions (where $L$ is enriched with finitely many relation symbols) was then carried out and announced (Vienna, June 2009), but remained unpublished, mostly because of the belief that it may not correspond to the right setting. It is only after subsequent work that precompact relational expansions appeared. For the reasons exposed above, we feel that they do provide the right setting. Future research will probably help deciding whether it is really the case or not. 

\bibliographystyle{amsalpha}
\bibliography{bib}

\providecommand{\bysame}{\leavevmode\hbox to3em{\hrulefill}\thinspace}
\providecommand{\MR}{\relax\ifhmode\unskip\space\fi MR }
% \MRhref is called by the amsart/book/proc definition of \MR.
\providecommand{\MRhref}[2]{%
  \href{http://www.ams.org/mathscinet-getitem?mr=#1}{#2}
}
\providecommand{\href}[2]{#2}
\begin{thebibliography}{LNVTS10}

\bibitem[Bar11]{B}
D.~Barto\v{s}ov\'a, \emph{Universal minimal flows of groups of automorphisms of
  uncountable structures}, preprint, 2011.

\bibitem[BP11]{BP}
M.~Bodirsky and M.~Pinsker, \emph{Reducts of {R}amsey structures}, Model
  theoretic methods in finite combinatorics (AMS, ed.), AMS Contemporary
  Mathematics, vol. 558, 2011, p.~31.

\bibitem[Che98]{Ch}
G.~L. Cherlin, \emph{The classification of countable homogeneous directed
  graphs and countable homogeneous {$n$}-tournaments}, Mem. Amer. Math. Soc.
  \textbf{131} (1998), no.~621, xiv+161.

\bibitem[GM83]{GM}
M.~Gromov and V.~D. Milman, \emph{A topological application of the
  isoperimetric inequality}, Amer. J. Math. \textbf{105} (1983), no.~4,
  843--854.

\bibitem[GW02]{GW1}
E.~Glasner and B.~Weiss, \emph{Minimal actions of the group {$\Bbb S(\Bbb Z)$}
  of permutations of the integers}, Geom. Funct. Anal. \textbf{12} (2002),
  no.~5, 964--988.

\bibitem[GW03]{GW2}
\bysame, \emph{The universal minimal system for the group of homeomorphisms of
  the {C}antor set}, Fund. Math. \textbf{176} (2003), no.~3, 277--289.

\bibitem[Jas10]{J}
J.~Jasi\'nski, \emph{Ramsey degrees of boron tree structures}, preprint, 2010.

\bibitem[KPT05]{KPT}
A.~S. Kechris, V.~G. Pestov, and S.~Todorcevic, \emph{Fra\"\i ss\'e limits,
  {R}amsey theory, and topological dynamics of automorphism groups}, Geom.
  Funct. Anal. \textbf{15} (2005), no.~1, 106--189.

\bibitem[Lac84]{La}
A.~H. Lachlan, \emph{Countable homogeneous tournaments}, Trans. Amer. Math.
  Soc. \textbf{284} (1984), no.~2, 431--461.

\bibitem[LNVTS10]{LNS}
C.~Laflamme, L.~Nguyen Van~Th{{\'e}}, and N.~W. Sauer, \emph{Partition
  properties of the dense local order and a colored version of {M}illiken's
  theorem}, Combinatorica \textbf{30} (2010), no.~1, 83--104.

\bibitem[Moo11]{M}
J.~Moore, \emph{Amenability and {R}amsey theory}, preprint, 2011.

\bibitem[MT11]{MT}
J.~Melleray and T.~Tsankov, \emph{Extremely amenable groups via continuous
  logic}, preprint, 2011.

\bibitem[Ne{\v{s}}07]{Ne}
J.~Ne{\v{s}}et{\v{r}}il, \emph{Metric spaces are {R}amsey}, European J. Combin.
  \textbf{28} (2007), no.~1, 457--468.

\bibitem[NVT10]{NVT1}
L.~Nguyen Van~Th{\'e}, \emph{Structural {R}amsey theory of metric spaces and
  topological dynamics of isometry groups}, Mem. Amer. Math. Soc. \textbf{206}
  (2010), no.~968, x+140.

\bibitem[Pes98]{Pe1}
V.~G. Pestov, \emph{On free actions, minimal flows, and a problem by {E}llis},
  Trans. Amer. Math. Soc. \textbf{350} (1998), no.~10, 4149--4165.

\bibitem[Pes02]{Pe0}
\bysame, \emph{Ramsey-{M}ilman phenomenon, {U}rysohn metric spaces, and
  extremely amenable groups}, Israel J. Math. \textbf{127} (2002), 317--357.

\bibitem[Pes06]{Pe}
\bysame, \emph{Dynamics of infinite-dimensional groups}, University Lecture
  Series, vol.~40, American Mathematical Society, Providence, RI, 2006, The
  Ramsey-Dvoretzky-Milman phenomenon, Revised edition of {\it Dynamics of
  infinite-dimensional groups and Ramsey-type phenomena} [Inst. Mat. Pura. Apl.
  (IMPA), Rio de Janeiro, 2005].

\bibitem[Sok10]{So0}
M.~Soki\'c, \emph{Ramsey properties of finite posets and related structures},
  Ph.D. thesis, University of Toronto, 2010.

\bibitem[Sok12a]{So1}
\bysame, \emph{Ramsey properties of finite posets}, Order \textbf{29} (2012),
  no.~1, 1--30.

\bibitem[Sok12b]{So2}
\bysame, \emph{Ramsey properties of finite posets ii}, Order \textbf{29}
  (2012), no.~1, 31--47.

\bibitem[Woo76]{W}
R.~E. Woodrow, \emph{Theories with a finite set of countable models and a small
  language}, Ph.D. thesis, Simon Fraser University, 1976.

\end{thebibliography}

%\begin{thebibliography}{NVT}

%\end{thebibliography}

\end{document}